\documentclass[12pt,journal,onecolumn]{IEEEtran}

\usepackage{graphicx}          
\usepackage[dvips]{epsfig}    
\usepackage{mathptmx} 
\usepackage{amsmath} 
\usepackage{amssymb}  
\usepackage{colordvi,multicol}
\usepackage{epstopdf}

\newtheorem{theorem}{Theorem}
\newtheorem{lemma}[theorem]{Lemma}
\newtheorem{remark}{Remark}[section]

\newtheorem{assumption}{Assumption}

\begin{document}

\title{Stochastic  Averaging in Discrete Time
and  Its Applications to Extremum Seeking}

\author{Shu-Jun Liu and Miroslav Krstic
\thanks{S. -J. Liu is with  the Department of Mathematics,
Southeast University, Nanjing, China.
        {\tt\small  sjliu@seu.edu.cn.} }%
\thanks{M. Krstic is with the Department of Mechanical and
Aerospace Engineering, University of California, San Diego, La
Jolla, CA 92093-0411, USA,
        {\tt\small krstic@ucsd.edu.}} }
\date{}
\maketitle
\baselineskip=1.6 \normalbaselineskip

\begin{abstract}
We investigate  stochastic averaging theory for locally Lipschitz discrete-time nonlinear systems with stochastic perturbation and its applications to convergence analysis of discrete-time stochastic extremum seeking  algorithms. Firstly, by defining two average systems (one is continuous time, the other is discrete time), we develop discrete-time stochastic averaging theorem  for locally Lipschitz nonlinear systems with stochastic perturbation. Our results only need some simple and applicable conditions, which are easy to verify, and remove a significant restriction present in existing results: global Lipschitzness of the nonlinear vector field. Secondly, we provide a discrete-time stochastic extremum seeking algorithm for a static map, in which measurement noise is considered and an ergodic discrete-time stochastic process is used as the excitation signal. Finally, for discrete-time nonlinear dynamical systems, in which the output equilibrium map has an extremum, we present a discrete-time stochastic extremum seeking scheme and, with a singular perturbation reduction, we prove the stability of the reduced system.
Compared with classical stochastic approximation methods, while the convergence that we prove is in a weaker sense, the conditions of the algorithm  are easy to verify and no requirements (e.g., boundedness) are imposed on the algorithm itself.
 \end{abstract}

\section{Introduction}
The averaging method is a powerful and elegant asymptotic analysis technique for nonlinear time-varying dynamical systems. Its basic idea is to
 approximate the original system (time-varying and
periodic or almost periodic, or randomly perturbed) by a simpler
(average) system (time-invariant, deterministic) or some
approximating diffusion system (a stochastic system simpler than the
original one). Averaging method has received intensive interests in the analysis of nonlinear dynamical systems (\cite{BaiFuSas88,BogMit61,FreWen84,RobSpa86,Sko89,ZhuYan97,Kha02,SkoHopSal02,SanVerMur07}),   adaptive control or adaptive algorithms
 (\cite{SasBod89,SoloKong95}),  and optimization methods (\cite{BenMetPri90,Che03,KusYin03,Spall03}).

Extremum seeking is a non-model based real-time optimization tool and also a method of adaptive control. Since the first proof of the convergence of extremum seeking \cite{KrsWan00}, the research on extremum seeking has triggered considerable interest in the theoretical control community (\cite{TeePop01,ChoKrsAriLee02,TanNesMar06,StaSti09,MoaManBre10,StaSti10,KhoTanManNes13}) and in applied communities (\cite{OuXuSch07,PopJanMagTee06,SchTorXu06}). According the choice of probing signals, the research on the extremum seeking method can be simply classified into two types: deterministic ES method
(\cite{ChoKrsAriLee02,AriKrs03,TanNesMar06,StaSti09,MoaManBre10,StaSti10}) and stochastic ES method (\cite{ManKrs09, LiuKrsTAC10,LiuKrsAuto10}). In the deterministic ES, periodic (sinusoidal) excitation signals are primarily used to probe the nonlinearity and estimate its gradient.  The random trajectory is preferable in some source  tasks where the orthogonality requirements on the elements of the periodic perturbation vector pose an implementation challenge for high dimensional systems.  Thus there is merit in investigating the use of stochastic perturbations within the ES architecture (\cite{LiuKrsTAC10}).

In \cite{LiuKrsTAC10}, we establish a framework of continuous-time stochastic extremum seeking algorithms by developing general stochastic averaging theory
in continuous time. However, there exists a need to consider stochastic extremum seeking in discrete time due to computer implementation.  Discrete-time extremum seeking with stochastic perturbation is investigated without measurement noise in \cite{ManKrs09}, in which the convergence of the algorithm involves strong restrictions on the iteration process. In \cite{StaSti09} and \cite{StaSti10}, discrete-time extremum seeking with sinusoidal perturbation is studied with measurement noise considered and the proof of the convergence is based on the classical idea of stochastic approximation method, in which the boundedness of iteration sequence is assumed to guarantee the convergence of the algorithm.

In this paper, we investigate stochastic averaging  for a
class of discrete-time locally Lipschitz nonlinear systems with stochastic
perturbation and then present discrete-time stochastic extremum seeking algorithm. In the first part, we develop general discrete-time stochastic averaging theory by the following four steps:
(i) we introduce two average systems: one is discrete-time average system, the other is
continuous-time average system; (ii) by a time-scale transformation, we establish a general stochastic averaging principle between the continuous-time average system and the original system in the continuous-time form; (iii) With the help of
the continuous-time average system, we establish stochastic averaging principle between the discrete-time average system and the original system; (iv) we establish
 some related stability theorems for the original system. To the best of our knowledge, this is the first work about discrete-time stochastic averaging for locally Lipschitz nonlinear systems.

 In the second part,
 we investigate general discrete-time stochastic extremum seeking with stochastic perturbation and measurement noise. We supply discrete-time stochastic extremum seeking algorithm for a static map and analyze  stochastic extremum seeking scheme for nonlinear dynamical systems with output equilibrium map. With the help of our developed discrete-time stochastic averaging theory, we  prove the convergence of the algorithms. Unlike in the continuous-time case \cite{LiuKrsTAC10}, in this work we consider the measurement noise, which is assumed to be bounded and ergodic stochastic process. In the classical stochastic approximation method, boundedness condition or other restrictions are imposed on the iteration algorithm itself to achieve the convergence with probability one. In our  stochastic discrete-time algorithm, the convergence condition is only imposed on the cost function or considered systems and is easy to verify, but as a consequence, we obtain a weaker form of convergence. Different from \cite{KhoTanManNes13} in which unified frameworks are proposed for extremum seeking of general nonlinear plants based on a sampled-data control law, we use the averaging method to analyze the stability of estimation error systems and avoid to verify the decaying property with a $\mathcal{K L}$ function of iteration sequence (the output sequence of extremum seeking controller), but we need justify the stability of average system.

The remainder of the paper is organized as follows. In Section \ref{sec-problem}, we give problem formulation of discrete-time stochastic averaging. In Section \ref{sec-resultsofSA} we establish our discrete-time stochastic  averaging theorems, whose proofs are given in the Appendix. In Section \ref{sec-applicationSM} we present stochastic  extremum seeking algorithms for a static map. In Section \ref{sec-application-dynamic}, we give stochastic extremum seeking scheme for dynamical systems and its stability analysis. In Section \ref{sec-conclusion} we offer some concluding remarks.

\paragraph*{Notation} $C_0(\mathbb{R}^n)$ denotes the family of all  continuous functions on $\mathbb{R}^n$ with compact supports. $[x]$ denotes the largest integer less than $x$.

\section{Problem formulation of discrete-time stochastic averaging}\label{sec-problem}
Consider system
\begin{align}\label{sys-1}
X_{k+1}=X_k+\epsilon f(X_k,Y_{k+1}), \ \ k=0,1,2,\ldots,
\end{align}
where $X_k\in \mathbb{R}^n$ is the state, $\{Y_{k}\}\subseteq \mathbb{R}^m$ is a stochastic perturbation  sequence defined on a complete probability space
 $(\Omega,\mathcal{F},P),$ where $\Omega$ is the sample space, $\mathcal{F}$ is the $\sigma$-field, and $P$ is the probability measure. Let $S_Y\subset\mathbb{R}^m$ be the living space of the perturbation process. $\epsilon\in (0,\epsilon_0)$ is a small parameter for some fixed positive constant $\epsilon_0$.

The following assumptions will be considered.¡¡
¡¡

\begin{assumption}\label{local-linear}
 The vector field $f(x,y)$ is a continuous function of
$(x,y)$,  and for any $x\in \mathbb{R}^n$,  it is a bounded
function of $y$. Further it satisfies the locally Lipschitz
condition in
  $x\in \mathbb{R}^n$ uniformly in $y\in S_Y$, i.e.,  for any compact
  subset $D\subset \mathbb{R}^n$, there is
   a constant $k_{D}$ such that for all $x_1,x_2\in D$
   and all $y\in S_Y$,
   \begin{align*}
   |f(x_1,y)-f(x_2,y)|\leq k_{D}\ |x_1-x_2|.
   \end{align*}
\end{assumption}

 \begin{assumption}\label{ergodic}
 The perturbation
 process
$\{Y_k\}$ is ergodic with invariant distribution $\mu$.
\end{assumption}

Under  Assumption \ref{ergodic}, we define two classes of average system of
system (\ref{sys-1}) as follows:
\begin{align}\label{avesys-1}
\mbox{Discrete average system:\ }\ & 
\bar{X}^{\rm d}_{k+1}=\bar{X}^{\rm d}_k+\epsilon \bar{f}(\bar{X}^{\rm d}_k), \ k=0,1,\ldots,\\
\mbox{Continuous average system:\ }\ &
\frac{d\bar{X}^{\rm c}(t)}{dt}=\bar{f}(\bar{X}^{\rm c}(t)),\ t\geq 0,\label{avesys-20}
\end{align}
where $\bar{X}^{\rm d}_0=\bar{X}^{\rm c}(0)=X_0$ and
\begin{align}\label{a-bar}
\bar{f}(x)
\triangleq \hspace{-1mm}\int_{S_Y} f(x,y) \mu(dy)=\hspace{-1mm}\lim_{N\to +\infty}\frac{1}{N+1}\sum_{k=0}^{N}f(x,Y_{k+1}) \mbox{ a.s. }
\end{align}
By Assumption \ref{local-linear}, $f(x,y)$ is bounded with respect to $y$, thus $y\to f(x,y)$ is $\mu$-integrable, so $\bar{f}$ is well defined. Here the definition of discrete average system is different from that in \cite{SoloKong95}, where the average vector field is defined by $\bar{f}(x)\triangleq Ef(x,Y_{k+1})$ (there, the perturbation process
$\{Y_{k+1}\}$ is assumed to be strict stationary). In this paper,  we consider ergodic process as perturbation.
It is easy to find discrete-time ergodic processes, e.g.,
\begin{itemize}
 \item i.i.d random variable sequence;
  \item finite state irreducible and aperiodic Markov process;
  \item $\{Y_i,i=0,1,\ldots,\}$ where $\{Y_t,t\geq 0\}$ is an Ornstein-Uhlenbeck (OU) process. In fact, for any continuous-time ergodic process $\{Y_t,t\geq 0\}$, the subsequence $\{Y_i,i=0,1,\ldots,\}$ is a discrete-time ergodic process.
\end{itemize}

For discrete average system (\ref{avesys-1}), the solution can be obtained by iteration,
thus the existence and uniqueness of the solution can be guaranteed by the local Lipschitzness of nonlinear vector field. For continuous average system (\ref{avesys-20}), $\bar{f}(x)$ is easy to be verified to be locally Lipschitz since $f(x,y)$ is locally Lipschitz in $x$. Thus, there exists a unique solution on $[0,\sigma_{\infty})$, where $\sigma_{\infty}$ is the explosion time. Thus we only need the following assumption.

\begin{assumption}\label{existence-ave2}
 The continuous average system (\ref{avesys-20}) has a
solution on $[0,+\infty)$. \end{assumption}

By  (\ref{sys-1}), we have
 \begin{align}\label{12}
X_{k+1}
=X_0+\epsilon\sum_{i=0}^kf(X_i,Y_{i+1}).
\end{align}
We introduce a new time $t_k= \epsilon k$. Denote $m(t)=\max\{k: t_k\leq t\}$ and define
$X(t)$ as a piecewise constant version of $X_k$, i.e.,
\begin{align}
X(t)=X_k, \mbox{ as } t_k\leq t< t_{k+1},
\end{align}
and $Y(t)$ as
a piecewise constant version of $Y_n$, i.e.,
\begin{align}
Y(t)=Y_k, \mbox{ as } t_k\leq t< t_{k+1}.
\end{align}
Then we can write (\ref{sys-1}) in the following form:
\begin{align}\label{sys-11}
X(t)&=X_0+\epsilon \sum_{k=1}^{m(t)}f(X_{k-1},Y_k)
\end{align}
or as the continuous-time version
\begin{align}\label{sys-2}
X(t)=X_0+ \int_0^t f(X(s),Y(\epsilon+s))ds-\int^t_{t_{m(t)}}f(X(s),Y(\epsilon+s))ds.
\end{align}
Similarly, we can write the discrete average system (\ref{avesys-1}) in the following continuous-time version
\begin{align}\label{avesys-2}
\bar{X}^{\rm d}(t)=X_0+ \int_0^t \bar{f}(\bar{X}^{\rm d}(s))ds-\int^t_{t_{m(t)}}\bar{f}(\bar{X}^{\rm d}(s))ds,
\end{align}
and write the continuous average system (\ref{avesys-20}) by
\begin{align}\label{avesys-2'}
\bar{X}^{\rm c}(t)=X_0+ \int_0^t \bar{f}(\bar{X}^{\rm c}(s))ds,
\end{align}
where $\bar{X}^{\rm d}(t)$ is a piecewise constant version of $\bar{X}^{\rm d}_k$, i.e.,
$\bar{X}^{\rm d}(t)=\bar{X}_k^{\rm d},$ as $t_k\leq t< t_{k+1}$.
We now rewrite the continuous-time version (\ref{sys-2}) of the original system (\ref{sys-1}) as two forms:
\begin{align}\label{sys-2-a}
X(t)=&X_0+ \int_0^t \bar{f}(X(s))ds-\int^t_{t_{m(t)}}\bar{f}(X(s))ds+
R^{(1)}(t,X(\cdot),Y(\epsilon+\cdot)),\\
X(t)=&X_0+ \int_0^t \bar{f}(X(s))ds+R^{(2)}(t,X(\cdot),Y(\epsilon+\cdot)),\label{sys-2-b}
\end{align}
where
\begin{align*}
R^{(1)}(t,X(\cdot),Y(\epsilon+\cdot))&=\int_0^{t_{m(t)}} \left(f(X(s),Y(\epsilon+s))
-\bar{f}(X(s))\right)ds,\\
R^{(2)}(t,X(\cdot),Y(\epsilon+\cdot))&=\int_0^{t} \left(f(X(s),Y(\epsilon+s))-\bar{f}(X(s))\right)ds\cr
&\quad-\int^t_{t_{m(t)}}f(X(s),Y(\epsilon+s))ds.
\end{align*}
Hence we  consider system (\ref{sys-2-a}) as a random perturbation of the continuous-time version (\ref{avesys-2}) of discrete average system (\ref{avesys-1}) and consider system (\ref{sys-2-b}) as a random perturbation of the continuous average system (\ref{avesys-2'}).

To study the solution property of the original system (\ref{sys-1}), we develop discrete-time stochastic averaging principle, i.e., using average systems (\ref{avesys-1}) or (\ref{avesys-20}) to approximate the original system (\ref{sys-1}).

Different to some existing discrete-time stochastic averaging results \cite{SoloKong95},
we consider averaging results under some weaker conditions: (a) the nonlinear vector field is locally Lipschitz; (b) the perturbation process is ergodic without other limitations. Under these weaker conditions, we can obtain weaker approximation results. The main idea is as follows.
First, we use the solution of continuous average system (\ref{avesys-20}) to approximate the solution of continuous-time version (\ref{sys-2}) of the original system (\ref{sys-1}), and then prove that the solution of  the continuous-time version (\ref{avesys-2}) of the discrete average system (\ref{avesys-1}) and the solution of the continuous average system (\ref{avesys-20}) are close to each other as small parameter $\epsilon$ is sufficiently small. Thus we can obtain that the discrete average system (\ref{avesys-1}) can approximate the original system (\ref{sys-1}) by transforming the continuous-time scale back to discrete-time scale. The main idea can be simplified as Fig. \ref{fig-mainidea} ($|\cdot|\to 0$ means the convergence in some probability sense).

\begin{figure}[!htbp]
\centering
\includegraphics*[width=7cm]{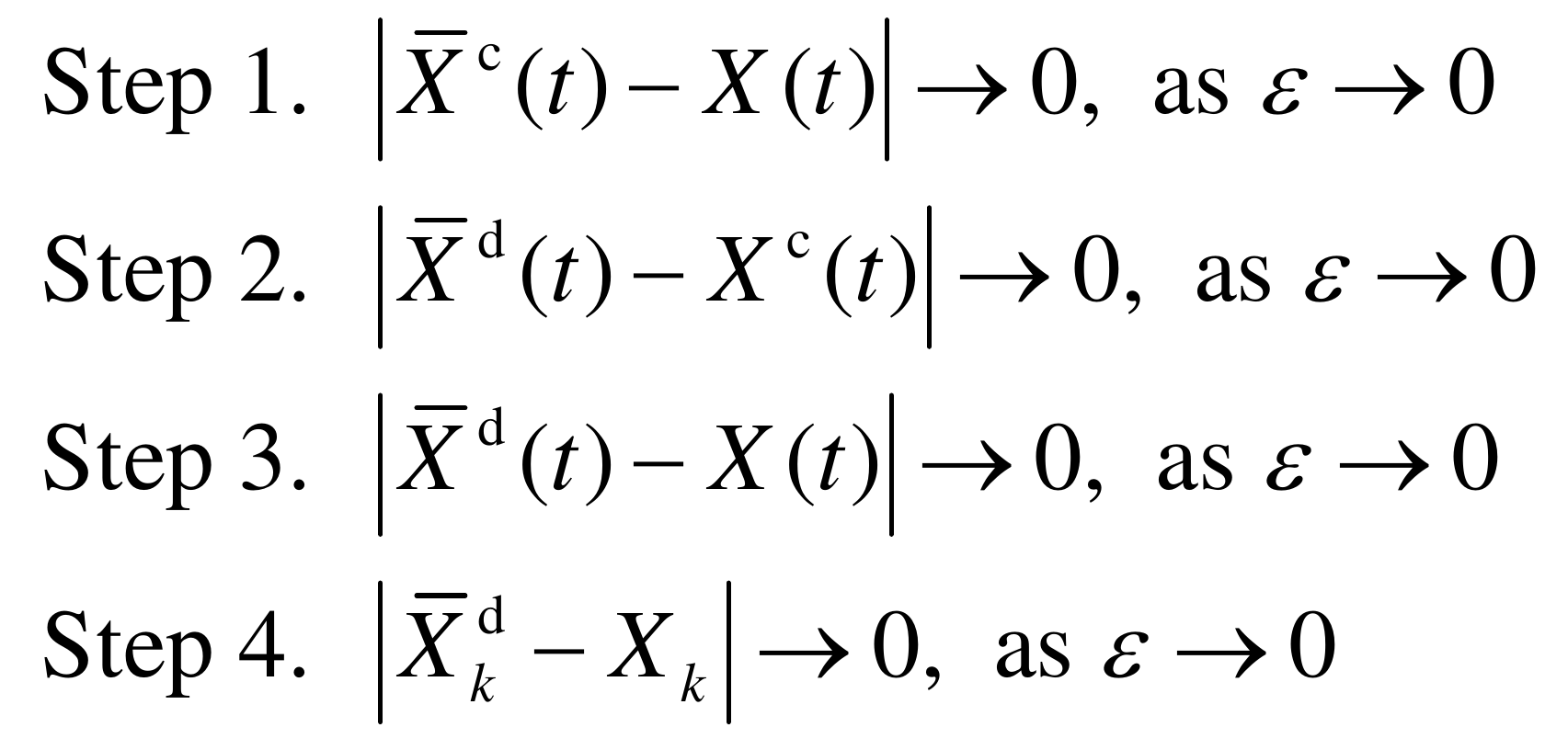}%
\hspace{1cm}%
\caption{Main idea for the proof of discrete-time stochastic averaging.}
 \label{fig-mainidea}
\end{figure}

\begin{remark}\label{remark-noise}
Our developed averaging theory is also applicable to the following systems
\begin{equation}\label{remark-sys1}
X_{k+1}=X_k+\epsilon \left( f(X_k,Y_{k+1})+W_{k+1}\right),\  \  k=0,1,2,\ldots,
\end{equation}
where $\{W_k\}\subseteq \mathbb{R}^n$ is   bounded with a bound $M$ and ergodic stochastic sequence, which is independent of the perturbation sequence $\{Y_{k}\}$.

Take a function $g\in C_0(\mathbb{R})$ such that $g(x)=1,$ for any $x\in B_M(0)=\{x\in\mathbb{R}^n:|x|\leq M\}$ and denote  $F(X_k,Z_{k+1})\triangleq f(X_k,Y_{k+1})+g(W_{k+1})$. Then we obtain the following system
\begin{equation}\label{remark-sys2}
X_{k+1}=X_k+\epsilon F(X_k,Z_{k+1}),\  \  k=0,1,2,\ldots.
\end{equation}
Since $\{W_k\in\mathbb{R}^n\}$ and $\{Y_{k}\}$  are independent and ergodic, we can obtain the combination process $Z_k\triangleq\{(Y_k^T,W_k^T)^T\}$ is also ergodic. It is
 easy to check that  the new system (\ref{remark-sys2}) satisfies Assumption \ref{local-linear}. Thus we know system (\ref{remark-sys1}) is included into our considered system (\ref{sys-1}).
\end{remark}

\section{Statements of General Results on Discrete-time Stochastic Averaging}\label{sec-resultsofSA}
Denote $X(t)$, $\bar{X}^{\rm c}(t)$, $\bar{X}^{\rm d}(t)$,
 as the solutions of system (\ref{sys-11})(the same as (\ref{sys-2})), continuous average system (\ref{avesys-2'}), and continuous-time version (\ref{avesys-2}) of discrete average system (\ref{avesys-1}), respectively. To avoid  too complex  mathematical symbols, we do not manifest $X(t)$ and $\bar{X}^{\rm c}(t)$ to be dependent on
 small parameter $\epsilon$.

 We have the
following results:
\begin{lemma}\label{lemma1}
Consider continuous-time version (\ref{sys-2}) of the original system under Assumptions \ref{local-linear},
\ref{ergodic} and \ref{existence-ave2}. Then
 for any $T>0$, 
\begin{align}\label{lem-a}
\lim_{\epsilon\to 0}\sup_{0\leq t\leq
T}|X(t)-\bar{X}^{\rm c}(t)|=0\ \ \mbox{a.s.}
\end{align}
\end{lemma}

\begin{proof}
See Appendix \ref{sec-pro-lem1}.
\end{proof}

\begin{lemma}\label{lemma2}
Consider continuous-time version (\ref{sys-2}) of the original system under Assumptions \ref{local-linear} and
\ref{ergodic}. Then
 for any $T>0$, 
\begin{align}\label{lem-b}
\lim_{\epsilon\to 0}\sup_{0\leq t\leq
T}|X(t)-\bar{X}^{\rm d}(t)|=0\ \mbox{a.s.}
\end{align}
\end{lemma}

\begin{proof}
See Appendix \ref{sec-pro-lem2}.
\end{proof}
Lemmas \ref{lemma1} and \ref{lemma2} supply finite-time approximation results in the sense of almost sure convergence. In \cite{GulVer93}, for any $\delta>0$, $\lim_{\epsilon \to 0}$ $P(\sup_{0\leq t\leq T}|X(t)-\bar{X}^{\rm c}(t)|>\delta)$ $=0$ or $\lim_{\epsilon \to 0}$ $E\sup_{0\leq t\leq T}|X(t)-\bar{X}^{\rm c}(t)|$ $=0$ can be achieved, but the nonlinear system is required to be globally Lipschitz. In this work, we present approximation result for locally Lipschitz systems.

\subsection{Approximation and stability results with continuous average system}
In this subsection, we present approximation results to the original system and the related stability results with the help of continuous average system. First, we extend the finite-time approximation result in Lemma~\ref{lemma1} to arbitrarily long time intervals.

\begin{theorem}\label{thm3}
Consider system (\ref{sys-2}) under Assumptions \ref{local-linear},
\ref{ergodic}  and \ref{existence-ave2}. Then

(i) for any
$\delta>0$, 
\begin{align}\label{thm3-a}
\lim_{\epsilon\to 0}\ \inf\{t\geq 0:
|X(t)-\bar{X}^{\rm c}(t)|>\delta\}=+\infty\ \ \mbox{ a.s.}
\end{align}

(ii)
there exists a function $T(\epsilon):(0,\epsilon_0)\to
\mathbb{N}$ such that for any $\delta>0$,
\begin{eqnarray}\label{thm3-b}
\lim_{\epsilon\to 0}\ P\left\{\sup_{0\leq t\leq
T(\epsilon)}|X(t)-\bar{X}^{\rm c}(t)|>\delta\right\}=0
\end{eqnarray}
with
\begin{eqnarray}\label{thm3-c}
     \lim_{\epsilon\to 0}T(\epsilon)=+\infty.
\end{eqnarray}
\end{theorem}

\begin{proof}
See Appendix \ref{sec-pro-thm3-a} for the result (\ref{thm3-a}), and Appendix \ref{sec-pro-thm3-b} for the result (\ref{thm3-b}).
\end{proof}

Similar to the continuous case \cite{LiuKrsTAC10}, although it is not assumed that the original system (\ref{sys-1}) or its continuous-time version (\ref{sys-2}) has an equilibrium,  average systems may have stable equilibria. We consider continuous-time version (\ref{sys-2}) of the original system  as a perturbation of continuous average system (\ref{avesys-20}),  and analyze weak stability properties by studying equilibrium stability of continuous average system (\ref{avesys-20}).
%

\begin{theorem}\label{thm4}
Consider  continuous-time version (\ref{sys-2}) of the original system under Assumptions
\ref{local-linear}, \ref{ergodic}
and \ref{existence-ave2}. Then
if the equilibrium $\bar{X}^{\rm c}(t)\equiv0$ of
continuous average system (\ref{avesys-20}) is  exponentially stable, then it
is weakly exponentially stable under random perturbation
$R^{(2)}(t,X(\cdot),Y(\epsilon+\cdot))$, i.e., there exist constants $r>0$, $c>0$
  $\gamma>0$ and a function $T(\epsilon):(0,\epsilon_0)\to
\mathbb{N}$ such that  for any initial condition $\bar{X}^{\rm c}(0)=X_0=x\in
\{\check{x}\in \mathbb{R}^n||\check{x}|<r\}$, and any $\delta>0$,
the solution of system (\ref{sys-2}) satisfies
 \begin{align}\label{thm4-5}
\lim_{\epsilon\to 0}\ \inf\left\{t\geq 0:
|X(t)|>c|x|e^{-\gamma t}+\delta\right\}=+\infty\ \mbox{ a.s.}
\end{align}
and   
%
%
%
%
%
\begin{eqnarray}\label{thm4-f}
\lim_{\epsilon\to 0}\ P\left\{|X(t)|\leq c|x|e^{-\gamma
t}+\delta,\forall t\in [0,T(\epsilon)]\right\}=1
\end{eqnarray}
 with $\lim_{\epsilon\to0}T(\epsilon)=+\infty$.
\end{theorem}

\begin{proof}
See Appendix \ref{sec-pro-thm4}.
\end{proof}

\begin{remark}
By analyzing the weak stability of continuous average system, we obtain  the solution property of system (\ref{sys-2}). Here the term ``weakly'' is used because the properties in question involve $\lim_{\epsilon\to 0}$ and are defined through the first exist time from a set. In \cite{Kha80}, stability concepts that are similarly defined under random perturbations are introduced for a nonlinear system perturbed by a stochastic process.
\end{remark}

\subsection{Approximation and stability results with continuous-time version of discrete average system}

Similarly, we can extend the finite-time approximation result in Lemma~\ref{lemma2} to arbitrarily long time intervals.

\begin{theorem}\label{thm5}
Consider  continuous-time version  (\ref{sys-2}) of the original system (\ref{sys-1}) under Assumptions \ref{local-linear} and
\ref{ergodic}. Then

(i) for any
$\delta>0$, 
\begin{align}\label{thm5-a}
\lim_{\epsilon\to 0}\ \inf\{t\geq 0:
|X(t)-\bar{X}^{\rm d}(t)|>\delta\}=+\infty\ \mbox{ a.s.};
\end{align}

(ii)
 there exists a function $T(\epsilon):(0,\epsilon_0)\to
\mathbb{N}$ such that for any $\delta>0$,
\begin{eqnarray}\label{thm5-b}
\lim_{\epsilon\to 0}\ P\left\{\sup_{0\leq t\leq
T(\epsilon)}|X(t)-\bar{X}^{\rm d}(t)|>\delta\right\}=0
\end{eqnarray}
with
     $\lim_{\epsilon\to 0}T(\epsilon)=+\infty.$
\end{theorem}

\begin{proof}
See Appendix \ref{sec-pro-thm5}.
\end{proof}

Similar as the last subsection, if the continuous-time version (\ref{avesys-2}) of discrete average system (\ref{avesys-1}) has some stability property, we can obtain the solution property of the continuous-time version (\ref{sys-2}) of the original system (\ref{sys-1}).

\begin{theorem}\label{thm6}
Consider  continuous-time version  (\ref{sys-2}) of the original system (\ref{sys-1}) under Assumptions
\ref{local-linear} and  \ref{ergodic}. Then
if for any $\epsilon\in (0,\epsilon_0)$, the equilibrium $\bar{X}^{\rm d}(t)\equiv0$ of continuous-time version (\ref{avesys-2}) of
 discrete average system (\ref{avesys-1}) is exponentially stable, then it
is weakly exponentially stable under random perturbation
$R^{(1)}(t,X(\cdot),Y(\epsilon+\cdot))$, i.e., there exist constants $r_{\epsilon}>0$, $c_{\epsilon}>0$, $\gamma_{\epsilon}>0$, and a function $T(\epsilon): (0,\epsilon_0)\to \mathbb{N}$ such that  for any initial condition $\bar{X}^{\rm d}_0=X_0=x\in
\{\check{x}\in \mathbb{R}^n:|\check{x}|<r_{\epsilon}\}$, and any $\delta>0$,
the solution of system (\ref{sys-2}) satisfies
 \begin{align}\label{thm6-6}
\lim_{\epsilon\to 0}\ \inf\left\{t\geq 0:
|X(t)|>c_{\epsilon}|x|e^{-\gamma_{\epsilon} t}+\delta\right\}=+\infty\ \mbox{ a.s.}
\end{align}
 and  
%
%
%
%
\begin{eqnarray}\label{thm6-f}
\lim_{\epsilon\to 0}\ P\left\{|X(t)|\leq c_{\epsilon}|x|e^{-\gamma_{\epsilon} t}+\delta,\forall t\in [0,T(\epsilon)]\right\}=1
\end{eqnarray}
 with   $\lim_{\epsilon\to 0}T(\epsilon)=+\infty$.
If the equilibrium $\bar{X}^{\rm d}(t)\equiv0$ of  continuous-time version (\ref{avesys-2}) of
 discrete average system (\ref{avesys-1}) is exponentially stable uniformly w.r.t. $\epsilon\in (0,\epsilon_0)$, then the above constants  $r_{\epsilon}, c_{\epsilon}>0$
and  $\gamma_{\epsilon}$ can be taken independent of $\epsilon$.
\end{theorem}

\begin{proof}
See Appendix \ref{sec-pro-thm6}.
\end{proof}

\begin{remark}\label{remark-thm6}
Notice that continuous average system (\ref{avesys-20}) is independent of small parameter $\epsilon$, but
discrete average system (\ref{avesys-1}) or its continuous-time version (\ref{avesys-2})
is dependent on $\epsilon$. The stability results (\ref{thm6-6})-(\ref{thm6-f}) are different from (\ref{thm4-5})-(\ref{thm4-f}), but in the discrete-time scale, the results are the same in essence, which can be seen in the next subsection.
\end{remark}

\subsection{Approximation and stability results in the discrete-time scale}

 In the last two subsections, we consider the approximation and stability analysis
 in the continuous time scale. Since the original system is in the discrete time scale, now we consider averaging results in the discrete time scale.

 Let $\{X_k\}$ and $\{\bar{X}^{\rm d}_k\}$ be the solutions of the original  system (\ref{sys-1}),  discrete average system
(\ref{avesys-1}),  respectively. Rewrite system (\ref{sys-1}) as
\begin{align}\label{sys-1-a}
X_{k+1}=X_k+\epsilon \bar{f}(X_k)+R^{(3)}(X_k,Y_{k+1}), \ \ k=0,1,2,\ldots,
\end{align}
where
$R^{(3)}(X_k,Y_{k+1})=\epsilon(f(X_k,Y_{k+1})-\bar{f}(X_k)).$
Hence we can consider system (\ref{sys-1-a}) (i.e. system (\ref{sys-1}))  as a random perturbation of discrete average system (\ref{avesys-1}). Thus we can analyze the weak stability of (\ref{avesys-1}) to obtain the solution property of (\ref{sys-1-a}) (i.e. system (\ref{sys-1})).

We have the following approximation results.
\begin{lemma}\label{lem7}
 Consider system (\ref{sys-1}) under Assumptions \ref{local-linear} and \ref{ergodic}. Then for any  $N\in\mathbb{N}$,
\begin{align}\label{lem7-a}
\lim_{\epsilon\to 0}\sup_{0\leq k\leq
[N/\epsilon]}|X_k-\bar{X}^{\rm d}_k|=0\ \mbox{ a.s.}
\end{align}
\end{lemma}

\begin{proof}
See Appendix \ref{sec-pro-lem7}.
\end{proof}

\begin{theorem}\label{thm8}
Consider system (\ref{sys-1}) under Assumptions \ref{local-linear} and \ref{ergodic}.
Then we have

(i) for any
$\delta>0$,
\begin{align}\label{thm8-a}
\lim_{\epsilon\to 0}\ \inf\{k\in\mathbb{N}:
|X_k-\bar{X}^{\rm d}_k|>\delta\}=+\infty\ \mbox{ a.s.};
\end{align}

(ii) for any $\delta>0$ and any  $N\in\mathbb{N}$,
\begin{eqnarray}\label{thm8-b}
\lim_{\epsilon\to 0}\ P\left\{\sup_{0\leq k\leq
[N/\epsilon]}|X_k-\bar{X}^{\rm d}_k|>\delta\right\}=0.
\end{eqnarray}
\end{theorem}

\begin{proof}
See Appendix \ref{sec-pro-thm8}.
\end{proof}

To the best of our knowledge, this is the  first approximation result of stochastic averaging in discrete time for locally Lipschitz systems. \cite{SoloKong95} developed discrete-time averaging theory, which mainly focuses on globally Lipschitz systems.

About the solution property of the original system (\ref{sys-1}) by analyzing the stability of  discrete average systems (\ref{avesys-1}), we have the following results.

\begin{theorem}\label{thm9}
Consider system (\ref{sys-1}) under Assumptions \ref{local-linear} and \ref{ergodic}. Then 
if for any $\epsilon\in (0,\epsilon_0)$, the equilibrium $\bar{X}^{\rm d}_k\equiv0$ of
the discrete average system (\ref{avesys-1}) is exponentially stable, then it
is weakly exponentially stable under random perturbation
$R^{(3)}(X_k,Y_{k+1})$, i.e., there exist constants $r_{\epsilon}>0$, $c_{\epsilon}>0$
and  $0<\gamma_{\epsilon}<1$ such that  for any initial condition $X_0=x\in
\{\check{x}\in \mathbb{R}^n:|\check{x}|<r_{\epsilon}\}$, and any $\delta>0$,
the solution of system (\ref{sys-1}) satisfies
 \begin{align}\label{thm9-6}
\lim_{\epsilon\to 0}\ \inf\left\{k\in\mathbb{N}:
|X_k|>c_{\epsilon}|x|(\gamma_{\epsilon})^k+\delta\right\}=+\infty\ \mbox{ a.s.}
\end{align}
and
\begin{eqnarray}\label{thm9-f}
\lim_{\epsilon\to 0}\ P\left\{|X_k|\leq c_{\epsilon}|x|(\gamma_{\epsilon})^k+\delta,\forall k=0,1,\ldots,[N/\epsilon]\right\}=1,
\end{eqnarray}
where $N$ is any natural number. If the equilibrium $\bar{X}^{\rm d}_k\equiv0$ of
 discrete average system (\ref{avesys-1}) is exponentially stable uniformly w.r.t. $\epsilon\in (0,\epsilon_0)$, then the above constants  $r_{\epsilon}, c_{\epsilon}>0$
and  $\gamma_{\epsilon}$ can be taken independent of $\epsilon$.

%
%
%
%
%
\end{theorem}

\begin{proof}
See Appendix \ref{sec-pro-thm9}.
\end{proof}

\begin{remark}\label{remark-otherstability}
Similar to continuous-time stochastic averaging in \cite{LiuKrsTAC10}, if the discrete average system has
other properties (boundedness, attractivity, stability and asymptotic stability),
we can obtain corresponding (boundedness, attractivity, stability and asymptotic stability) results for the original system.
\end{remark}

\begin{remark}\label{remark-global}
The stability results in the above Theorems \ref{thm4}, \ref{thm6}, and \ref{thm9} are local, but if the equilibrium of  corresponding average systems  is globally stable (asymptotical stable or exponentially stable),  then it is globally weakly stable (asymptotical stable or exponentially stable, respectively).
\end{remark}

\begin{remark}\label{remark-SESandSA}
 In analyzing the similar kind of system as (\ref{sys-1}), average method is  different from ordinary differential equation (ODE) method and weak convergence method of stochastic approximation (\cite{Lju77,Che03,KusYin03}). In ODE method (gain coefficients are changing with iteration steps) and weak convergence method (constant gain), regression function or cost function (i.e., $f(\cdot)$ in system (\ref{sys-1})) acts as the nonlinear vector field  of  an ordinary differential equation, which is used to compare with the original system. In both methods, to analyze the convergence of the solution of system (\ref{sys-1}),
there need some restrictions on both the growth rate of nonlinear vector fields and the algorithm itself (i.e., uniformly bounded), or the existence of some continuously differential function (Lyapunov function) satisfying some conditions. In some sense, these conditions (i.e. boundedness of iteration sequence) are not easy to verify.
Average method is to use a new system (namely, average system) to approximate system (\ref{sys-1}) and thus makes it possible to obtain the solution property of system (\ref{sys-1}). From Theorems \ref{thm8} and \ref{thm9}, we can see that the approximation conditions  are easy to verify.
\end{remark}

\section{Discrete-time Stochastic Extremum seeking Algorithm for Static Map}\label{sec-applicationSM}
Consider the quadratic function
\begin{equation}\label{sm-function}
\varphi(x)=\varphi^*+\frac{\varphi''}{2}(x-x^*)^2,
\end{equation}
where $x^*\in \mathbb{R}$, $\varphi^*\in\mathbb{R}$, and $\varphi''\in\mathbb{R}$ are unknown. Any
$\mathbb{C}^2$ function $\varphi(\cdot)$
 with an extremum at $x=x^*$ and with $\varphi''\neq 0$ can be locally approximated
by (\ref{sm-function}). Without loss of generality, we assume that
$\varphi''>0$. In this section, we design an algorithm to make
$|x_k-x^*|$ as small as possible, so that the output $\varphi(x_k)$
 is driven to  its minimum $\varphi^*$. The only available information is the output $y=\varphi(x_k)$ with measurement noise.

Denote $\hat{x}_k$ as the $k$ step estimate of the unknown optimal input
$x^*$. Design iteration algorithm as
\begin{equation}\label{es-iteration}
\hat{x}_{k+1}=\hat{x}_k-\epsilon \sin(v_{k+1})y_{k+1},\ \ k=0,1,\ldots,
\end{equation}
where $y_{k+1}=\varphi(x_k)+W_{k+1}$ is the measurement output, $\{v_k\}$ is an ergodic stochastic process with invariant distribution $\mu$ and living
space $S_v$, and  $\{W_{k}\}$ is the measurement noise, which is assumed to be bounded with a bound $M$ and ergodic with invariant distribution $\nu$ and living space $S_W$. $\epsilon\in(0,\epsilon_0)$ is a positive small parameter for some constant $\epsilon_0>0$. The perturbation process $\{v_k\}$ is independent of the measurement noise process $\{W_k\}$.

Define $x_k=\hat{x}_k+a\sin(v_{k+1})$, $a>0$ and the estimation error $\tilde{x}_k=\hat{x}_k-x^*$. Then we have
\begin{align}\label{es-err-sys}
\tilde{x}_{k+1}=\tilde{x}_k-\epsilon \sin(v_{k+1})\left[\varphi^*+\frac{\varphi^{''}}{2}\left(\tilde{x}_k
+a\sin(v_{k+1})\right)^2+W_{k+1}\right].
\end{align}
To analyze the solution property of the error system (\ref{es-err-sys}), we will use stochastic averaging theory developed in the last section. First, to calculate the average system, we choose the excitation process $\{v_k\}$
 as a sequence of i.i.d. random variable  with invariant distribution   $\mu(dy)=\frac{1}{\sqrt{2\pi}\sigma}e^{-\frac{y^2}{2\sigma^2}}dy$. We assume that the measurement noise process $\{W_k\}$ is any bounded ergodic process.

By (\ref{a-bar}), we have
 \begin{align}
 {\rm Ave}\{\sin^{i}(v_{k+1})\}&\hspace{-1mm}\triangleq\hspace{-1mm}\int_{S_v}\sin^{i}(y)\mu(dy)=0,\quad i=1,3,\\
 {\rm Ave}\{\sin^2(v_{k+1})\}&\hspace{-1mm}\triangleq\hspace{-1mm}\int_{S_v}\sin^2(y)\mu(dy)=\frac12-\frac12e^{-2\sigma^2},\\
 {\rm Ave}\{\sin(v_{k+1})W_{k+1}\}&\hspace{-1mm}\triangleq\hspace{-1mm}\int_{S_v\times S_W}\sin(y)x\mu(dy)\times\nu(dx)\cr
 &\hspace{-1mm}=\hspace{-1mm}\int_{S_v}\sin(y)\mu(dy)\times \int_{S_W}x\nu(dx)=0.
 \end{align}
Thus, we obtain  the average system of the error system (\ref{es-err-sys})
\begin{align}\label{es-err-avesys}
\tilde{x}_{k+1}^{\rm ave}&=\tilde{x}_{k}^{\rm ave}-\epsilon \frac{a\varphi^{''}(1-e^{-2\sigma^2})}{2}\tilde{x}_{k}^{\rm ave}\cr
&=\left(1-\epsilon \frac{a\varphi^{''}(1-e^{-2\sigma^2})}{2}\right)\tilde{x}_{k}^{\rm ave}.
\end{align}
Since $\varphi''>0$, there exists $\epsilon^*=\frac{2}{a\varphi^{''}(1-e^{-2\sigma^2})}$ such that  the average system (\ref{es-err-avesys})
is globally exponentially stable for $\epsilon\in (0,\epsilon^*)$.

\begin{figure}[t]
\centering
\includegraphics*[width=7cm]{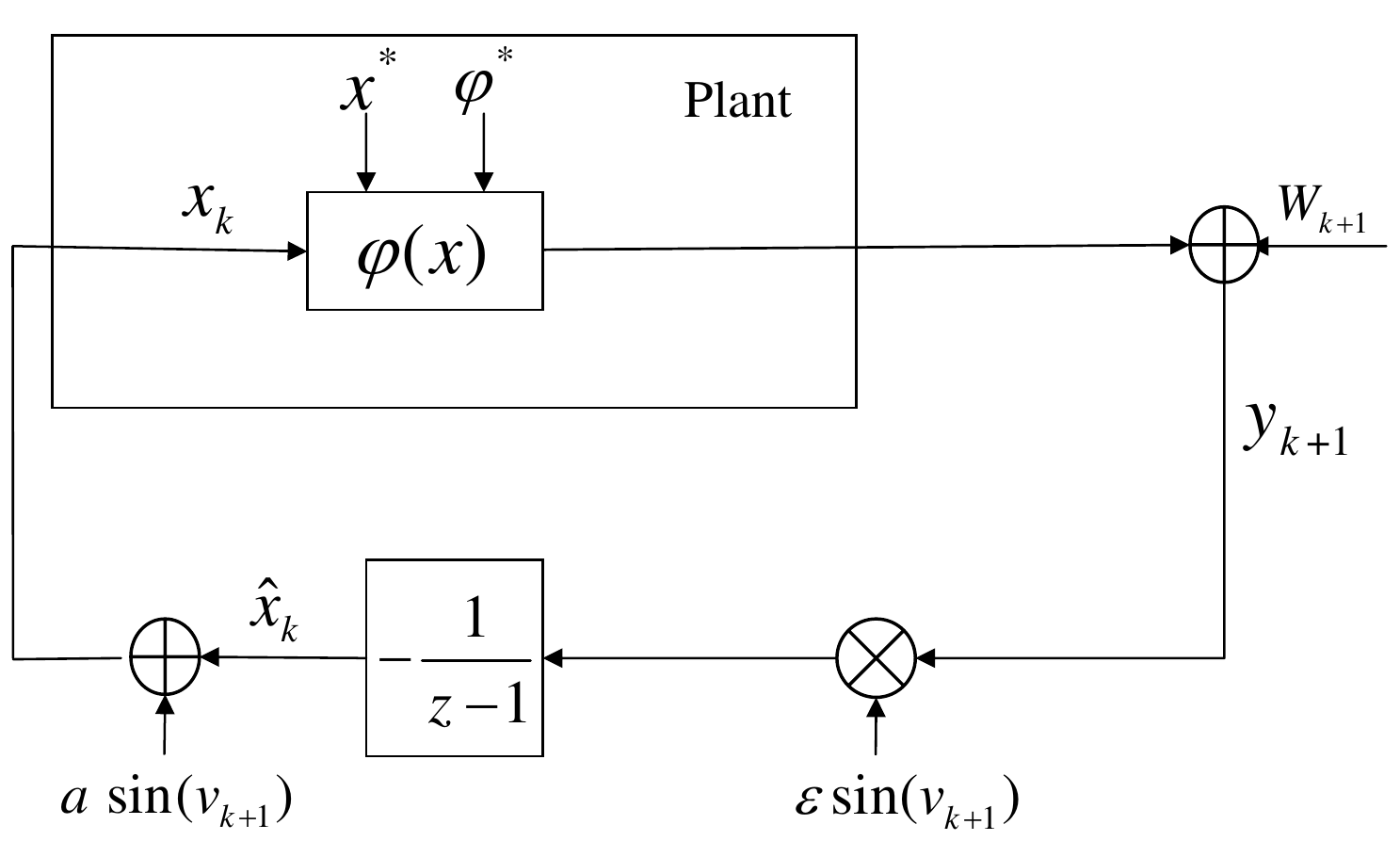}%
\hspace{1cm}%
\caption{Discrete-time stochastic extremum seeking scheme for a static map.}
 \label{fig-dis-es}
\end{figure}

Thus by  Theorem \ref{thm9}, Remark \ref{remark-noise} and Remark \ref{remark-global}, for the discrete-time stochastic extremum seeking
algorithm
 in Fig. \ref{fig-dis-es}, we have the following theorem.
\begin{theorem}\label{es-th-static}
Consider the static map (\ref{sm-function}) under the iteration algorithm  (\ref{es-iteration}). Then
 there exist constants $c_{\epsilon}>0$ and $0<\gamma_{\epsilon}<1$  such that for
  any initial condition
  $\tilde{x}_0\in \mathbb{R}$ and any
$\delta>0,$
\begin{eqnarray}\label{static-result-1}
\lim_{\epsilon\to0}\ \inf\left\{k\in\mathbb{N}:
\left|\tilde{x}_k\right|
>c_{\epsilon}\left|\tilde{x}_0\right|(\gamma_{\epsilon})^k+\delta\right\}=+\infty\  \mbox{ a.s. }
\end{eqnarray}
and
\begin{align}\label{static-result-2}
\lim_{\epsilon\to 0}P\left\{|\tilde{x}_k|\leq c_{\epsilon}|\tilde{x}_0|(\gamma_{\epsilon})^k+\delta,\forall k=0,1,\ldots,[N/\epsilon]\right\}=1.
\end{align}
\end{theorem}

These two results imply that the norm of the error vector
$\tilde{x}_k$ exponentially converges, both almost surely
and in probability, to below an arbitrarily small residual value
$\delta$ over an arbitrarily long time interval, which tends to
infinity as $\epsilon$ goes to zero. To quantify the output convergence to the extremum, for any $\epsilon>0$, define a stopping time
$$
\tau_{\epsilon}^{\delta}=\inf\left\{k\in \mathbb{N}:|\tilde{x}_k|
>c_{\epsilon}\left|\tilde{x}_0\right|(\gamma_{\epsilon})^k+\delta\right\}.
$$
Then by (\ref{static-result-1}), we know that $\lim\limits_{\epsilon\to
0}\tau_{\epsilon}^{\delta}=+\infty,\ a.s.$ and
\begin{equation}\label{static-erroranalysis}
\left|\tilde{x}_k\right|\leq
c_{\epsilon}\left|\tilde{x}_0\right|(\gamma_{\epsilon})^k+\delta, \ \ \forall
k< \tau_{\epsilon}^{\delta}.
\end{equation}
  Since $y_{k+1}=\varphi(x^*+\tilde{x}_k
  +a\sin(v_{k+1}))+W_{k+1}$
and $\varphi^{'}(x^{*})=0$, we have
\begin{align*}
y_{k+1}-\varphi(x^*)&=\frac{\varphi^{''}(x^*)}{2}(\tilde{x}_k
+a\sin(v_{k+1}))^2\cr
&\quad\quad+O\left((\tilde{x}_k+a\sin(v_{k+1}))^3\right)+W_{k+1}.
\end{align*}
Thus by (\ref{static-erroranalysis}) and $|W_{k+1}|\leq M$, it holds that
\begin{eqnarray*}\label{static-erroranalysis-a}
|y_{k+1}-\varphi(x^*)|\leq O(a^2)+O(\delta^2)+C_{\epsilon}\left|\tilde{x}_0\right|^2(\gamma_{\epsilon})^{2k}
+M,  \forall k<\tau_{\epsilon}^{\delta},
\end{eqnarray*}
 for some positive constant $C_{\epsilon}$. Similarly,
by (\ref{static-result-2}),
\begin{eqnarray}\label{static-erroranalysis-b}
\lim_{\epsilon\to0}P\hspace{-3mm}&\left\{|y_{k+1}-\varphi(x^*)|\leq
O(a^2)+O(\delta^2)+C_{\epsilon}\left|\tilde{x}_0\right|^2(\gamma_{\epsilon})^{2k},\right.\cr
&\left. +M, \ \forall k=0,1,\ldots,[N/\epsilon]\right\}=1,
\end{eqnarray}
which implies that the output can exponentially approach to the extremum $\varphi(x^*)$ if $a$
is chosen sufficiently small and the measurement noise can be ignored $(M=0)$. By (\ref{es-err-avesys}), we can see that smaller $\epsilon$ is, slower the convergence rate of the average error system is. Thus, for this static map, the  parameter $\epsilon$ is designed to consider the tradeoff of the convergence rate and convergence precision.

\begin{remark}\label{remark-sm}
As an optimization method,  besides the different derivative estimation methods, there are some other differences between stochastic extremum seeking (SES) and stochastic approximation (SA)(\cite{Spall03, StaSti10}). First, in the iteration, the gain coefficients in SA is often changing with the iteration step, but for SES, the gain coefficient
is a small constant and denotes the amplitude of the excitation signal; Second, stochastic approximation may consider more kinds of measurement noise (i.e., martingale difference sequence, some kind of infinite correlated sequence), but here we assume the measurement noise as bounded ergodic stochastic sequence (the boundedness is to guarantee the existence of the integral in (\ref{a-bar});  Third, to prove the convergence of the algorithm ($P\{\lim_{k\to\infty}x_k=x^*\}=1$) , SA algorithm  requires some restrictions on the cost function (regression function) or the iteration sequence, while the convergence conditions of SES algorithm are simple and easy to verify.
\end{remark}

\begin{figure}[t]
\centering
\includegraphics*[width=7cm]{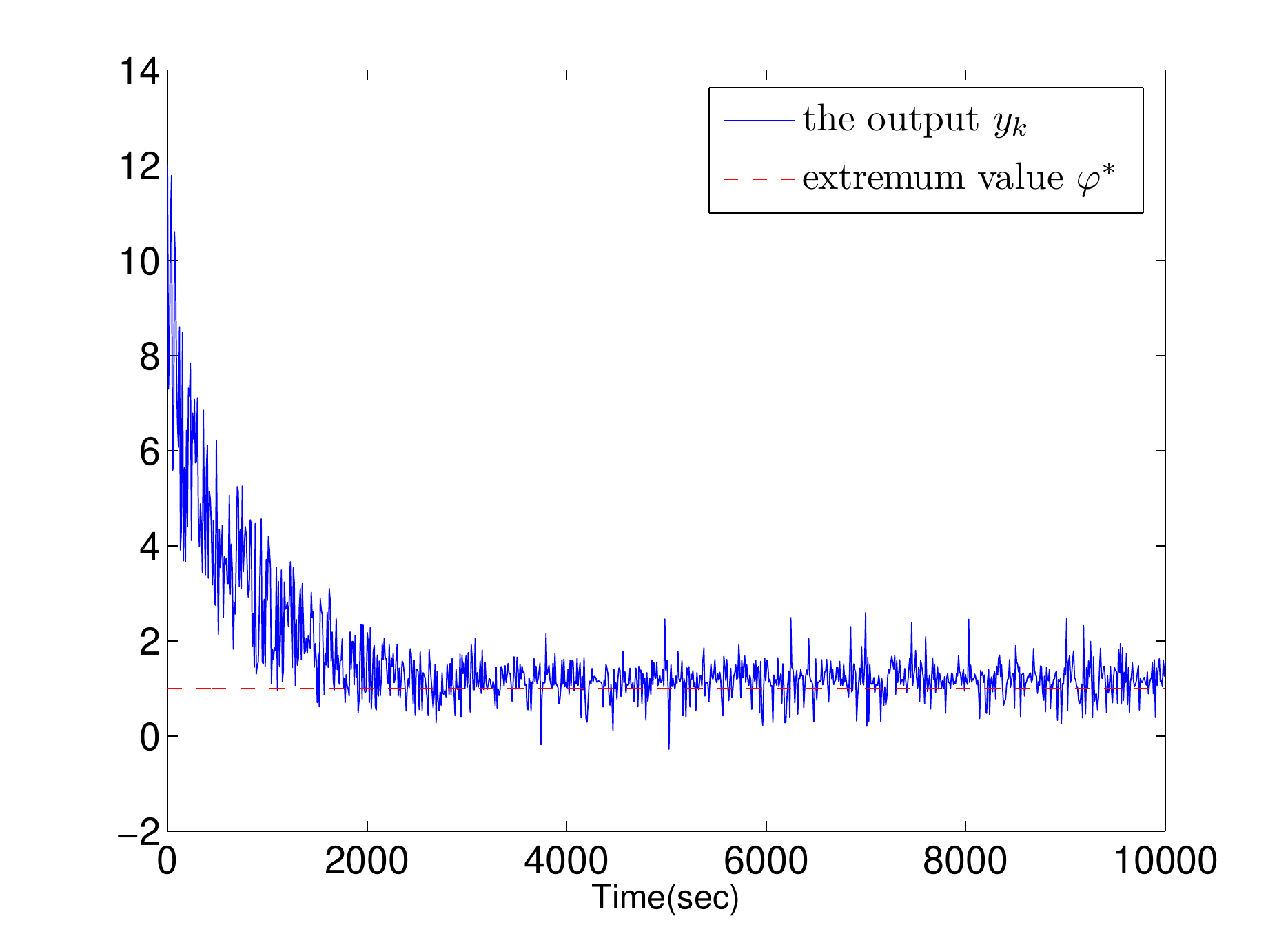}%
\hspace{8mm}%
\includegraphics*[width=7cm]{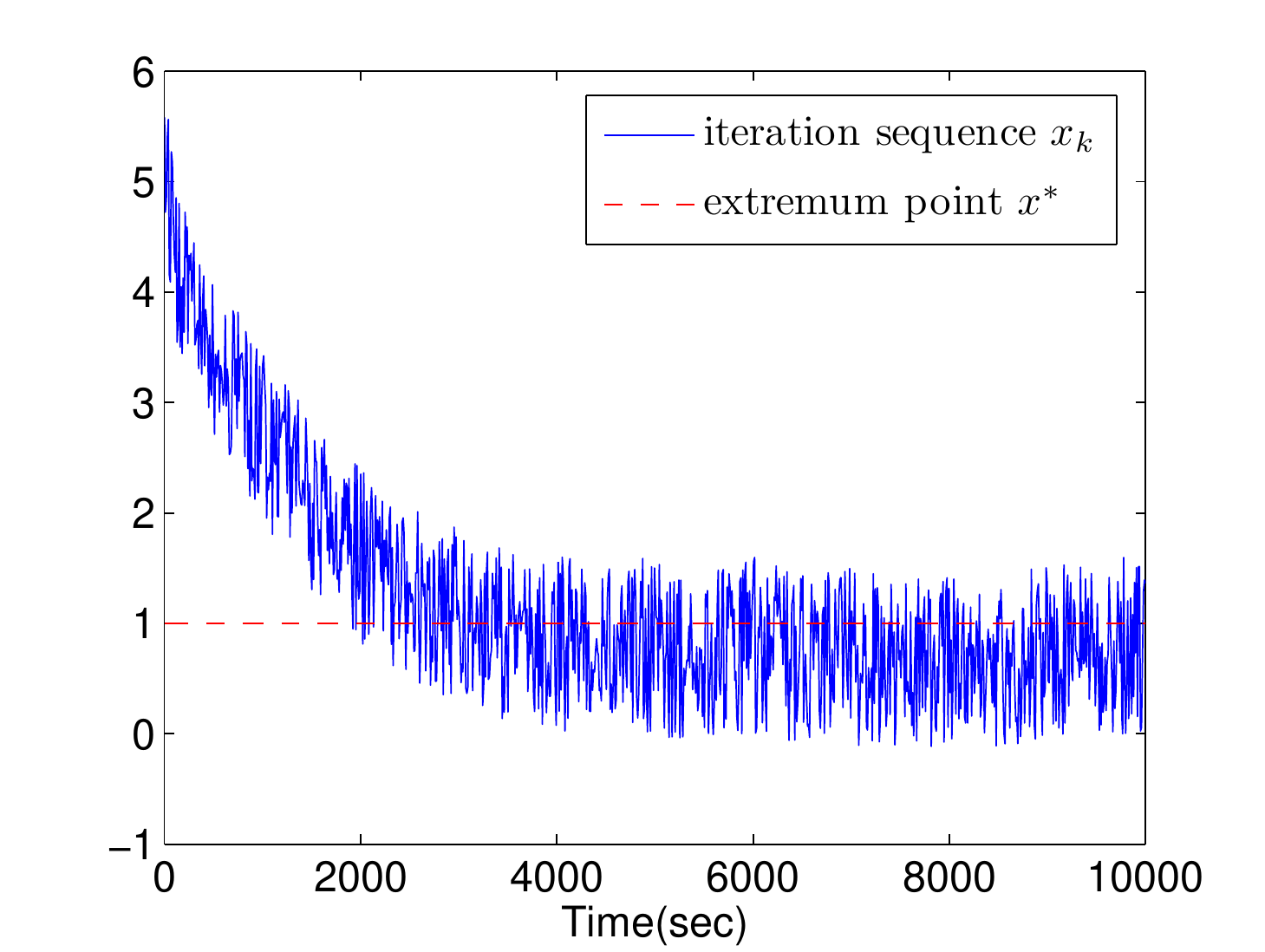}%
\caption{Discrete-time stochastic ES with independent Gaussian random variable sequence as the stochastic probing signal.}
 \label{fig-staticmap}
\end{figure}

Fig.\ref{fig-staticmap} displays the simulation results with $\varphi^*=1,\varphi''=1,x^*=1,$
in the static map (\ref{sm-function}) and $a=0.8,\epsilon=0.002$
in the parameter update law  (\ref{es-iteration}) and initial condition $\hat{x}_0=5$. The probing signal $\{v_k\}$ is taken as a sequence of i.i.d. gaussian random variables
with distribution $N(0,2^2)$
and the measurement noise $\{W_k\}$ is taken as a sequence of truncated i.i.d. gaussian random variables with distribution $N(0,0.2^2)$.

\section{Discrete-time Stochastic Extremum Seeking for Dynamic Systems}\label{sec-application-dynamic}
Consider a general  nonlinear
model
\begin{align}
x_{k+1}&=f(x_k,u_k),\\ y_k^0&=h(x_k), \ \ k=0,1,2,\ldots,
\end{align}
where $x_k\in\mathbb{R}^n$ is the state, $u_k\in\mathbb{R}$ is the
input, $y_k^0\in\mathbb{R}$ is the nominal output, and $f: \mathbb{R}^n\times
\mathbb{R}\to\mathbb{R}^n$ and $h:\mathbb{R}^n\to\mathbb{R}$ are
smooth functions. Suppose that we know a smooth control law
\begin{align}
u_k=\beta(x_k,\theta)
\end{align}
parameterized by a scalar parameter $\theta$. Then the closed-loop
system
\begin{eqnarray}\label{dyn-closedloop}
x_{k+1}=f(x_k,\beta(x_k,\theta))
\end{eqnarray}
 has equilibria parameterized by $\theta.$ We make the following assumptions about the
closed-loop system.

\begin{assumption}\label{dyn-es-A1}
 There exists a smooth function
$l:\mathbb{R}\to\mathbb{R}^n$ such that
\begin{align}
f(x_k,\beta(x_k,\theta))=x_k \ \ \mbox{ if and only if }\ \  x_k=l(\theta).
\end{align}
\end{assumption}

\begin{assumption}\label{dyn-es-A3}
 There exists $\theta^*\in\mathbb{R}$ such that
\begin{align}
(h\circ l)'(\theta^*)=0,\\
(h\circ l)''(\theta^*)<0.
\end{align}
\end{assumption}

Thus, we assume that the output equilibrium map $y=h(l(\theta))$ has
a local maximum at $\theta=\theta^*$.

Our objective is to develop a feedback mechanism which makes the
output equilibrium map $y=(h(l(\theta)))$ as close  as possible to
the maximum $y^*=h(l(\theta^*))$
 but without requiring the
knowledge of either $\theta^*$ or the functions $h$ and $l$. The only available information is the  output with measurement noise.

\begin{figure}[t]
\centering
\includegraphics*[width=7cm]{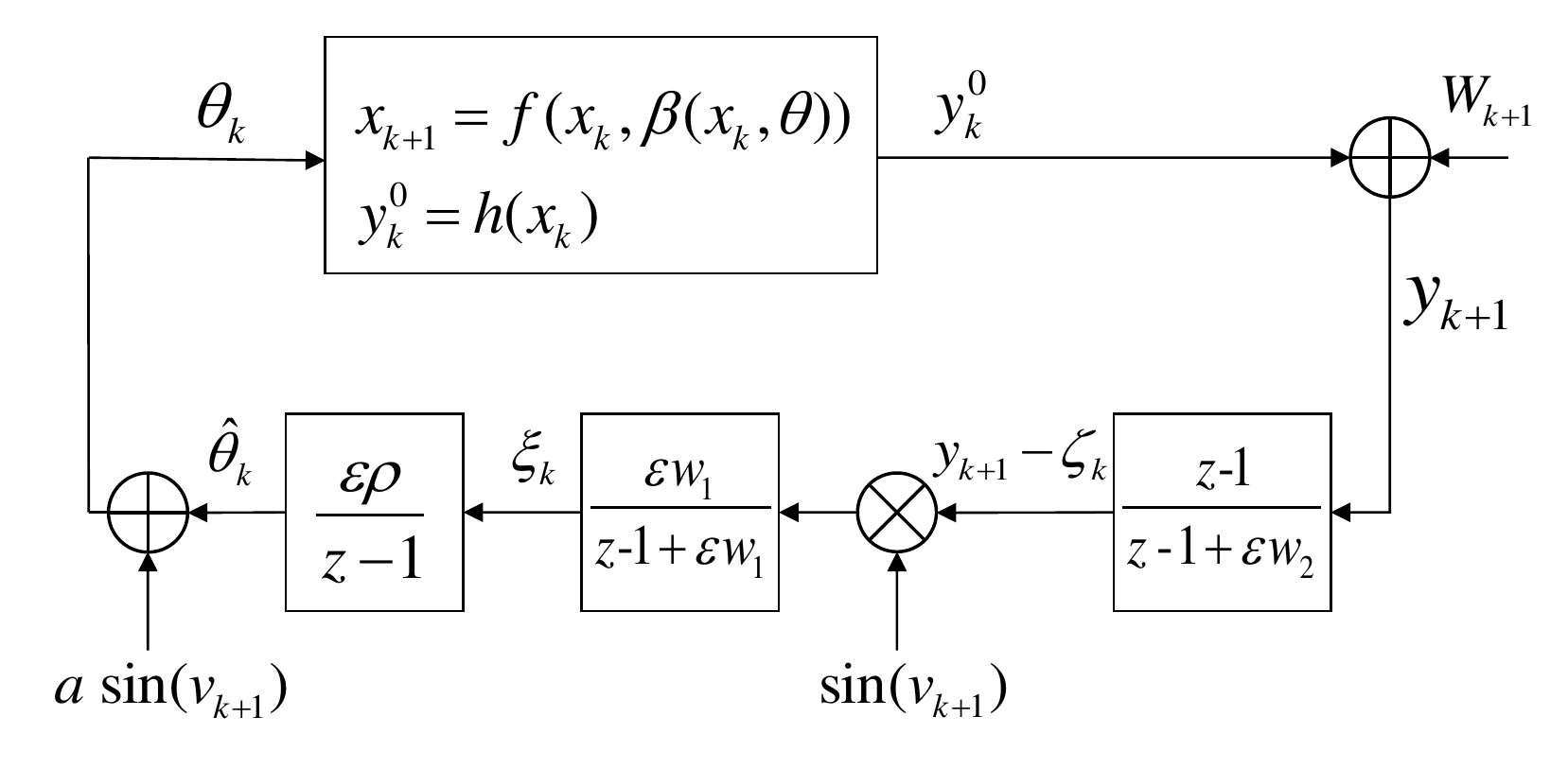}%
\hspace{1cm}%
\caption{Discrete-time stochastic extremum seeking scheme for nonlinear dynamics}
 \label{figure-dis-nonlinear}
\end{figure}

As discrete-time stochastic extremum seeking scheme in Fig. \ref{figure-dis-nonlinear}, we choose the parameter update law
\begin{align}
\hat{\theta}_{k+1}&=\hat{\theta}_k+\epsilon \varrho  \xi_k,\\
\xi_{k+1}&=\xi_k-\epsilon w_1\xi_k+\epsilon w_1(y_{k+1}-\zeta_k)\sin(v_{k+1}),\\
\zeta_{k+1}&=\zeta_k-\epsilon w_2\zeta_k+\epsilon w_2y_{k+1},\\
y_{k+1}&=y_k^0+W_{k+1},
\end{align}
where $\varrho>0,w_1>0,w_2>0,\epsilon>0$ are design parameters,
 $\{v_k\}$ is  assumed to be a sequence of i.i.d. Gaussian random variable  with  distribution $\mu(dx)=\frac{1}{\sqrt{2 \pi}\sigma}e^{-\frac{x^2}{2\sigma^2}}dx,$ and $W_k\triangleq(-M)\vee Z_k\wedge M$ is measurement noise, where $\{Z_k\}$ is  a sequence of i.i.d. Gaussian random variable  with distribution $\nu(dx)=\frac{1}{\sqrt{2 \pi}\sigma_1}e^{-\frac{x^2}{2\sigma_1^2}}dx$. We assume that the probing signal $\{v_k\}$  is independent of the measure noise process $\{W_k\}$. It is easy to verify that $\{W_k\}$ is a bounded and ergodic process with invariant distribution
\begin{equation}
\nu_1(A)=\nu(A\wedge(-M,M))+q_1+q_2, \mbox{ for any } \ A \subseteq \mathbb{R}
\end{equation}
where $q_1=\left\{
  \begin{array}{ll}
    \nu([M,+\infty)), & \mbox{ if } M\in A\\
    0, & \mbox{else}
  \end{array}
\right.$, and
$q_2=\left\{
  \begin{array}{ll}
    \nu((-\infty,-M]), & \mbox{ if } -M\in A \\
    0, & \mbox{ else}
  \end{array}
\right.$.

Define
\begin{align}
\theta_k=\hat{\theta}_k+a \sin(v_{k+1}).
\end{align}
Then we obtain the closed-loop system as
\begin{align}
x_{k+1}&=f(x_k,\beta(x_k,\hat{\theta}_k+a\sin(v_{k+1}))),\\
\hat{\theta}_{k+1}&=\hat{\theta}_k+\epsilon \varrho \xi_k,\\
\xi_{k+1}&=\xi_k-\epsilon w_1\xi_k+\epsilon w_1(y_k^0+W_{k+1}-\zeta_k)\sin(v_{k+1}),\\
\zeta_{k+1}&=\zeta_k-\epsilon w_2\zeta_k+\epsilon w_2(y_k^0+W_{k+1}).
\end{align}
With the error variable
\begin{align}
\tilde{\theta}_k&=\hat{\theta}_k-\theta^*,\\
\tilde{\zeta}_k&=\zeta_k-h(l(\theta^*)),
\end{align}
the closed-loop system is rewritten as
\begin{align}
x_{k+1}&=f(x_k,\beta(x_k,\tilde{\theta}_k+\theta^*+a\sin(v_{k+1}))),\label{reclosedloop-x}\\
\tilde{\theta}_{k+1}&=\tilde{\theta}_k+\epsilon \varrho \xi_k,\label{reclosedloop-theta}\\
\xi_{k+1}&=\xi_k-\epsilon w_1\xi_k+\epsilon w_1\left(h(x_k)-h(l (\theta^*))-\tilde{\zeta}_k+W_{k+1}\right)\sin(v_{k+1}),\label{reclosedloop-xi}\\
\tilde{\zeta}_{k+1}&=\tilde{\zeta}_k-\epsilon w_2\tilde{\zeta}_k+
\epsilon w_2\left(h(x_k)-h(l(\theta^*))+W_{k+1}\right).\label{reclosedloop-zeta}
\end{align}
We employ a singular perturbation reduction, freeze $x_k$ in (\ref{reclosedloop-x}) at its
quasi-steady state value as
\begin{equation}
x_k=l(\theta^*+\tilde{\theta}_k+a\sin(v_{k+1}))
\end{equation}
and substitute it into (\ref{reclosedloop-theta})-(\ref{reclosedloop-zeta}),
and then get the reduced system
\begin{align}
\tilde{\theta}^{\rm r}_{k+1}&=\tilde{\theta}^{\rm r}_k+\epsilon \varrho \xi^{\rm r}_k,\label{reducesys-theta}\\
\xi^{\rm r}_{k+1}&=\xi^{\rm r}_k-\epsilon w_1\xi^{\rm r}_k+\epsilon w_1\left(\varsigma(\tilde{\theta}^{\rm r}_k+a\sin(v_{k+1}))
-\tilde{\zeta}^{\rm r}_k+W_{k+1}\right)\sin(v_{k+1}),\label{reducesys-xi}\\
\tilde{\zeta}^{\rm r}_{k+1}&=\tilde{\zeta}^{\rm r}_k-\epsilon w_2\tilde{\zeta}^{\rm r}_k
+\epsilon w_2\left(\varsigma(\tilde{\theta}^{\rm r}_k+a\sin(v_{k+1}))+W_{k+1}\right).\label{reducesys-zeta}
\end{align}
where $\varsigma(\tilde{\theta}^{\rm r}_k+a\sin(v_{k+1}))\triangleq h\left(l(\theta^*+\tilde{\theta}^{\rm r}_k+a\sin(v_{k+1}))\right)-h(l(\theta^*))$.
With Assumption \ref{dyn-es-A3}, we have
\begin{align}
\varsigma(0)&=0,\label{v1}\\
\varsigma'(0)&=(h\circ l)'(\theta^*)=0,\label{v2}\\
\varsigma''(0)&=(h\circ l)''(\theta^*)<0.\label{v3}
\end{align}
Now we use our stochastic averaging theorems to analyze the reduced
system (\ref{reducesys-theta})-(\ref{reducesys-zeta}). According to (\ref{a-bar}),
  we obtain that the average system of (\ref{reducesys-theta})-(\ref{reducesys-zeta}) is
\begin{align}\label{avesys-nonlinear}
&\left[
                   \begin{array}{c}
                     \tilde{\theta}_{k+1}^{\rm r,ave}- \tilde{\theta}_{k}^{\rm r,ave}\\
                     \xi_{k+1}^{\rm r,ave}-\xi_{k}^{\rm r,ave}\\
                     \tilde{\zeta}_{k+1}^{\rm r,ave}-\tilde{\zeta}_{k}^{\rm r,ave}\\
                   \end{array}
                 \right] \cr
 & = \epsilon\left[
                   \begin{array}{c}
                     \varrho \xi_k^{\rm r,ave}\\
                    -w_1\xi_k^{\rm r,ave}+w_1
 \int_{S_v} \varsigma(\tilde{\theta}_k^{\rm r,ave}
 +a\sin(y))\sin(y)\mu(dy)\\
                     -w_2\tilde{\zeta}_k^{\rm r,ave}
 +w_2
 \int_{S_v} \varsigma(\tilde{\theta}_k^{\rm r,ave}
 +a\sin(y))\mu(dy)\\
                   \end{array}
                 \right],\cr
                 &
\end{align}
where we use the following facts:
\begin{equation}
\int_{S_W}x\nu_1(dx)=0,\quad\int_{S_W\times S_v}x\sin(y)\nu_1(dx)\times \mu(dy)=0.
\end{equation}
Now, we determine the average equilibrium $(\tilde{\theta}^{\rm a,e},\xi^{\rm a,e},\tilde{\zeta}^{\rm a,e})$ which satisfies
\begin{align}
\xi^{\rm a,e}&=0,\label{dyn-equ1}\\
 -w_1\xi^{\rm a,e}+w_1
 \int_{S_v} \varsigma(\tilde{\theta}^{\rm a,e}
 +a\sin(y))\sin(y)\mu(dy)&=0,\label{dyn-equ2}\\
  -w_2\tilde{\zeta}^{\rm a,e}
 +w_2
 \int_{S_v} \varsigma(\tilde{\theta}^{\rm a,e}
 +a\sin(y))\mu(dy)&=0.\label{dyn-equ3}
\end{align}
We assume that $\tilde{\theta}^{\rm a,e}$ has the form
\begin{align}
\tilde{\theta}^{\rm a,e}=b_1a+b_2a^2+O(a^3).\label{taylortheta}
\end{align}
By (\ref{v1}) and (\ref{v2}), define
\begin{align}
\varsigma(x)=\frac{\varsigma''(0)}{2}x^2+\frac{\varsigma'''(0)}{3!}x^3+O(x^4).\label{taylorvar}
\end{align}
Then substituting (\ref{taylortheta}) and (\ref{taylorvar}) into (\ref{dyn-equ2}) and noticing that $S_{v}=\mathbb{R}$, we
have
\begin{align}\label{calculate1}
&\int_{-\infty}^{+\infty}\varsigma(b_1a+b_2a^2+O(a^3)
+a\sin(y))\sin(y){1\over\sqrt{2\pi}\sigma}e^{-{y^2\over 2\sigma^2}}dy\nonumber\\
&= \int_{-\infty}^{+\infty}\left[{\varsigma''(0)\over 2}\left(
b_1a+b_2a^2+O(a^3)+a\sin(y)\right)^2\right.\cr
&\quad\left.+{\varsigma'''(0)\over
3!}\left(b_1a+b_2a^2+O(a^3)+a\sin(y)\right)^3\right.\cr
&\quad\left.+O((b_1a+b_2a^2+O(a^3)+a\sin(y))^4)\right]
\sin(y){e^{-{y^2\over
2\sigma^2}}\over\sqrt{2\pi}\sigma}dy\nonumber\\
 &=\int_{-\infty}^{+\infty}\left[{\varsigma''(0)\over 2}
(2b_1a^2+2b_2a^3+O(a^4))\sin^2(y)\right.\cr
&\quad\left.+{\varsigma'''(0)\over
3!}(3b_1^2a^3+O(a^4)
+a^3\sin^2(y))\sin^2(y)\right]\cr
&\quad\times{1\over\sqrt{2\pi}\sigma}e^{-{y^2\over
2\sigma^2}}dy+O(a^4)\cr
&=O(a^4)+\varsigma''(0)b_1\left(\frac12-\frac12e^{-2\sigma^2}\right)a^2\cr
&\quad+\left[\left(b_2\varsigma''(0)
+{\varsigma'''(0)\over
2}b_1^2\right)\left(\frac12-\frac12e^{-2\sigma^2}\right)\right.\cr
&\quad\left.+{\varsigma'''(0)\over
6}\left(\frac38-\frac12e^{-2\sigma^2}+\frac{1}{8}e^{-8\sigma^2}\right)\right]a^3
=0,
\end{align}
where the following facts are used:
\begin{align}
&{1\over
\sqrt{2\pi}\sigma}\int_{-\infty}^{+\infty}\sin^{2k+1}(y)e^{-{y^2\over
2\sigma^2}}dy=0,\ k=0,1,2, \ldots,\cr &{1\over
\sqrt{2\pi}\sigma}\int_{-\infty}^{+\infty}\sin^{2}(y)e^{-{y^2\over
2\sigma^2}}dy=\frac12-\frac12e^{-2\sigma^2},\cr &{1\over
\sqrt{2\pi}\sigma}\int_{-\infty}^{+\infty}\sin^{4}(y)e^{-{y^2\over
2\sigma^2}}dy=\frac38-\frac12e^{-2\sigma^2}+\frac{1}{8}e^{-8\sigma^2}.
\end{align}
Comparing the coefficients of the powers of $a$ on the right-hand
and left-hand sides of (\ref{calculate1}), we have
\begin{align}
b_1&=0,\\ b_{2}&=-{\varsigma'''(0)(3-4e^{-2\sigma^2}+e^{-8\sigma^2})\over
24\varsigma''(0)(1-e^{-2\sigma^2})},
\end{align}
and thus by (\ref{taylortheta}), we have
\begin{align}
\tilde{\theta}^{\rm a,e}=-{\varsigma'''(0)(3-4e^{-2\sigma^2}+e^{-8\sigma^2})\over
24\varsigma''(0)(1-e^{-2\sigma^2})}a^2+O(a^3).
\end{align}
From this equation, together with (\ref{dyn-equ3}), we have
\begin{align}
\tilde{\zeta}^{\rm a,e}&=\int_{-\infty}^{+\infty}\varsigma\left(\tilde{\theta}^{\rm a,e}
 +a\sin(y)\right){1\over \sqrt{2\pi}\sigma}e^{-{y^2\over 2\sigma^2}}dy\cr
 &=\int_{-\infty}^{+\infty}\varsigma\left(b_2a^2+O(a^3)
+a\sin(y)\right){e^{-{y^2\over 2\sigma^2}}\over \sqrt{2\pi}\sigma}dy\cr
 &=
\int_{-\infty}^{+\infty}\left[{\varsigma''(0)\over 2}\left(
b_2a^2+O(a^3)+a\sin(y)\right)^2\right.\cr
&\left.\quad+{\varsigma'''(0)\over
3!}\left(b_2a^2+O(a^3)+a\sin(y)\right)^3\right.\cr
&\quad\left.+O\left((b_2a^2+O(a^3)+a\sin(y))^4\right)\phantom{{\varsigma''(0)\over
2}}\hspace{-1cm}\right]{e^{-{y^2\over 2\sigma^2}}\over \sqrt{2\pi}\sigma}dy\cr
&={a^2\varsigma''(0)\over 2}\int_{-\infty}^{+\infty}\sin^2(y){1\over
\sqrt{2\pi}\sigma}e^{-{y^2\over 2\sigma^2}}dy+O(a^3)\cr
&={\varsigma''(0)(1-e^{-2\sigma^2})\over 4}a^2+O(a^3).
\end{align}
Thus the equilibrium of the average system (\ref{avesys-nonlinear})
is
\begin{align}
\left[
  \begin{array}{c}
    \tilde{\theta}^{\rm a,e} \\
    \xi^{\rm a,e} \\
    \tilde{\zeta}^{\rm a,e} \\
  \end{array}
\right]= \left[
  \begin{array}{c}
    -{\varsigma'''(0)(3-4e^{-2\sigma^2}+e^{-8\sigma^2})\over
24\varsigma''(0)(1-e^{-2\sigma^2})}a^2+O(a^3) \\
    0 \\
    {\varsigma''(0)(1-e^{-2\sigma^2})\over 4}a^2+O(a^3) \\
  \end{array}
\right].
\end{align}
The Jacobian matrix of the average system (\ref{avesys-nonlinear})
at the equilibrium
$(\tilde{\theta}^{\rm a,e},\xi^{\rm a,e},\tilde{\zeta}^{\rm a,e})$ is
\begin{eqnarray}\label{jacobi}
J_{\rm r}^{\rm a}=\left[
        \begin{array}{ccc}
          1 & \epsilon \rho & 0 \\
          \epsilon J_{\rm r21}^{\rm a}
           & 1-\epsilon w_1 & 0 \\
          \epsilon J_{\rm r31}^{\rm a}& 0 & 1-\epsilon w_2 \\
        \end{array}
      \right],
\end{eqnarray}
where
\begin{align}
J_{\rm r21}^{\rm a}&={w_1\over\sqrt{2\pi}\sigma}
          \int_{-\infty}^{+\infty}\varsigma'\left(\tilde{\theta}^{\rm a,e}
          +a\sin(y)\right)\sin(y)e^{-{y^2\over 2\sigma^2}}dy,\\
J_{\rm r31}^{\rm a}&={w_2\over\sqrt{2\pi}\sigma}\int_{-\infty}^{+\infty}\varsigma'
(\tilde{\theta}^{\rm a,e}
          +a\sin(y))e^{-{y^2\over 2\sigma^2}}dy.
\end{align}
Thus we have
\begin{align}\label{jacobi-det}
&det(\lambda I-J_{\rm r}^{\rm a})\cr
&=(\lambda-1+\epsilon w_2)\left((\lambda-1)^2+\epsilon w_1(\lambda-1)-\epsilon^2\rho J_{\rm r21}^{\rm a}\right).
\end{align}
With Taylor expansion and by calculating the integral, we get
\begin{align}\label{int}
\int_{-\infty}^{+\infty}&\varsigma'\left(\tilde{\theta}^{\rm a,e}+a\sin(y)\right)
          \sin(y)e^{-{y^2\over
          2\sigma^2}}dy\cr
         & =a\sqrt{2\pi}\sigma \varsigma''(0)\left(\frac12-\frac12e^{-2\sigma^2}\right)+O(a^2).
\end{align}
By substituting (\ref{int}) into (\ref{jacobi-det}) we get
\begin{align}
det(\lambda I-J_{\rm r}^{\rm a})&=(\lambda-1+\epsilon w_2)\left((\lambda-1)^2+\epsilon w_1(\lambda-1)\right.\cr
&\left.-\frac{\epsilon^2\rho w_1a}{2}\varsigma^{''}(0)(1-e^{-2\sigma^2})-\frac{\epsilon^2\rho w_1}{\sqrt{2\pi}\sigma}O(a^2) \right)\cr
&=(\lambda-1+\epsilon w_2)(\lambda-1-\Pi_1) (\lambda-1-\Pi_2).
\end{align}
where $\Pi_1=\epsilon\frac{-w_1+\sqrt{w_1^2+2\rho w_1 a \varsigma^{''}(0)(1-e^{-2\sigma^2})+\frac{4\rho w_1}{\sqrt{2\pi}\sigma}O(a^2)}}{2},$
$\Pi_2=\epsilon\frac{-w_1-\sqrt{w_1^2+2\rho w_1 a \varsigma^{''}(0)(1-e^{-2\sigma^2})+\frac{4\rho w_1}{\sqrt{2\pi}\sigma}O(a^2)}}{2}.$
Since $\varsigma^{''}(0)<0$, for sufficiently small $a$, $\sqrt{w_1^2+2\rho w_1 a \varsigma^{''}(0)(1-e^{-2\sigma^2})+\frac{4\rho w_1}{\sqrt{2\pi}\sigma}O(a^2)}$ can be smaller than $w_1$. Thus there exist $\epsilon_1^*>0$, such that for $\epsilon\in(0,\epsilon_1^*)$, the eigenvalues of the Jacobian matrix of the average system  (\ref{avesys-nonlinear}) are in the unit ball, and thus  the equilibrium of the average system is
exponentially stable. Then according to
Theorem \ref{thm9}, we have the following result for stochastic
extremum seeking algorithm in Fig. \ref{figure-dis-nonlinear}.

\begin{theorem}\label{theorem-nonlinear}
\ Consider the reduced system (\ref{reducesys-theta})-(\ref{reducesys-xi})-(\ref{reducesys-zeta}) under Assumption
\ref{dyn-es-A3}.
 Then there
exists a constant $a^*>0$ such that for any $0<a<a^*$
   there exist constants $r_{\epsilon}>0,c_{\epsilon}>0$ and $0<\gamma_{\epsilon}<1$
   such that for
  any initial condition
  $\left|\Delta_0^{\epsilon}\right|<r_{\epsilon}$, and any $\delta>0,$
\begin{eqnarray}\label{nonlinear-a}
\lim_{\epsilon\to0}\ \inf\left\{k\in\mathbb{N}: |\Delta^{\epsilon}_k|
>c_{\epsilon}|\Delta_0^{\epsilon}| (\gamma_{\epsilon})^k
+\delta\right\}=+\infty, \mbox{ a.s. }
\end{eqnarray}
and
\begin{eqnarray}\label{nonlinear-b}
\lim_{\epsilon\to0}P\left\{|\Delta^{\epsilon}_k|
\leq c_{\epsilon}|\Delta_0^{\epsilon}| (\gamma_{\epsilon})^k
+\delta,\forall k=0,1,\ldots,[N/\epsilon]\right\}=1,
\end{eqnarray}
 where
$\Delta^{\epsilon}_k\triangleq$
$(\tilde{\theta}_k^{\rm r},\xi_k^{\rm r},\tilde{\zeta}_k^{\rm r})$-
$\left(-{\varsigma'''(0)(3-4e^{-2\sigma^2}+e^{-8\sigma^2})\over
24\varsigma''(0)(1-e^{-2\sigma^2})}a^2+O(a^3),\right.$
   $ 0,$
   $ \left.{\varsigma''(0)(1-e^{-2\sigma^2})\over 4}a^2+O(a^3)\right)$, and $N$ is any natural number.
\end{theorem}

These results imply  that the norm of the error vector $\Delta^{\epsilon}_k$ exponentially converges, both almost surely and in probability, to below an arbitrarily small residual value $\delta$
over an arbitrary large time interval, which tends to infinity as
the perturbation parameter $\epsilon$ goes to zero. In particular,
the $\tilde{\theta}^{\rm r}_k$-component of the error vector
converges to below $\delta$. To quantify the output convergence to
the extremum,
we define a stopping time
$$
\tau_{\epsilon}^{\delta}=\inf\left\{k\in \mathbb{N}:|\Delta^{\epsilon}_k|
>c_{\epsilon}\left|\Delta^{\epsilon}_0\right|(\gamma_{\epsilon})^k+\delta\right\}.
$$
Then by (\ref{nonlinear-a}) and the definition of
$\Delta^{\epsilon}_k$, we know that $\lim\limits_{\epsilon\to
0}\tau_{\epsilon}^{\delta}=+\infty,\ a.s.$ and for all $k<\tau_{\epsilon}^{\delta}$
\begin{align}
&\left|\tilde{\theta}^{\rm r}_k-\left(-{v'''(0)(3-4e^{-q^2}+e^{-4q^2})\over
24v''(0)(1-e^{-q^2})}a^2+O(a^3)\right)\right|\leq
c_{\epsilon}\left|\Delta^{\epsilon}_0\right|(\gamma_{\epsilon})^k+\delta,
\end{align}
   which implies that
\begin{eqnarray}\label{erroranalysis}
   \left|\tilde{\theta}^{\rm r}_k\right|\leq O(a^2)+
c_{\epsilon}\left|\Delta^{\epsilon}_0\right|(\gamma_{\epsilon})^k+\delta, \ \ \
\forall k<\tau_{\epsilon}^{\delta}.
   \end{eqnarray}
Since the nominal output $y_k^0=h(l(\theta^*+\tilde{\theta}^{\rm r}_k+a\sin(v_{k+1})))$ and
$(h\circ l)^{'}(\theta^{*})=0$, we have
\begin{align*}
y_k^0-h(l(\theta^*))&=\frac{(h\circ l)^{''}(\theta^*)}{2}
(\tilde{\theta}^{\rm r}_k+a\sin(v_{k+1}))^2+O\left((\tilde{\theta}^{\rm r}_k+a\sin(v_{k+1}))^3\right).
\end{align*}
Thus by (\ref{erroranalysis}), it holds that
\begin{eqnarray*}
|y_k^0-h\circ l(\theta^*)|\leq O(a^2)+O(\delta^2)+C_{\epsilon}\left|\Delta^{{\epsilon}}_0\right|^2(\gamma_{\epsilon})^{2k}, \ \
\forall k< \tau_{\epsilon}^{\delta},
\end{eqnarray*}
 for some positive constant $C_{\epsilon}$. Similarly, by (\ref{nonlinear-b})
 \begin{eqnarray*}
\lim_{\epsilon\to0}\hspace{-3mm}&P\left\{|y_k^0-h\circ l(\theta^*)|\leq
O(a^2)+O(\delta^2)+C_{\epsilon}\left|\Delta^{{\epsilon}}_0\right|^2(\gamma_{\epsilon})^{2k},\right.\cr
&
\left.\forall k=0,1,\ldots,[N/\epsilon]\right\}=1.
\end{eqnarray*}
With the measurement noise considered, we obtain that
\begin{align*}
|y_{k+1}-h\circ l(\theta^*)|&\leq O(a^2)+O(\delta^2)+C_{\epsilon}\left|\Delta^{{\epsilon}}_0\right|^2(\gamma_{\epsilon})^{2k}\cr
&\quad\quad+M, \ \ \forall k< \tau_{\epsilon}^{\delta},\nonumber
\end{align*}
 for some positive constant $C_{\epsilon}$, and moreover,
 \begin{eqnarray*}
\lim_{\epsilon\to0}\hspace{-3mm}&P\left\{|y_{k+1}-h\circ l(\theta^*)|\leq
O(a^2)+O(\delta^2)+C_{\epsilon}\left|\Delta^{{\epsilon}}_0\right|^2(\gamma_{\epsilon})^{2k}\right.\cr
&\left.+M,\  \forall k=0,1,\ldots,[N/\epsilon]\right\}=1.
\end{eqnarray*}

\begin{remark}
For stochastic ES scheme for dynamical systems with output equilibrium map, we focus on the stability of the reduced system. Different from the deterministic ES case (periodic probing signal), the closed-loop system (\ref{reclosedloop-x})-(\ref{reclosedloop-zeta}) has two perturbations (small parameter $\epsilon$ and stochastic perturbation $\{v_{k}\}$) and thus generally, there is  no equilibrium solution or periodic solution. So we can not analyze the solution property of the closed-loop system by general singular perturbation methods for both deterministic systems (\cite{Kha02}) and stochastic systems (\cite{Soc98}). But for the reduced system (parameter estimation error system when the state is at its quasi-steady state value), we can analyze the solution property by our developed averaging theory to obtain the approximation to the maximum of output equilibrium map.
\end{remark}

\section{Concluding remarks}\label{sec-conclusion}
In this paper, we develop discrete-time stochastic averaging theory and apply it to analyze the convergence of  our proposed stochastic discrete-time extremum seeking algorithms. Our results of stochastic averaging extend the existing  discrete-time averaging theorems for globally Lipschitz systems to locally Lipschitz systems. Compared with other stochastic optimization methods, e.g., stochastic approximation, simulated annealing method and genetic algorithm, the convergence conditions of  discrete-time stochastic extremum seeking algorithm are easier to verify and clearer. Compared with continuous-time stochastic extremum seeking, in the discrete-time case, we consider the bounded measurement noise.
In our results, we can only prove the weaker convergence than the convergence with probability one of classical stochastic approximation. Better convergence of algorithms and improved algorithms are our future work directions. For dynamical systems, we only focus on the stability of parameter estimation error system at the quasi-steady state value (the reduced system). For the whole closed-loop system with extremum seeking controller, we will investigate the proper singular perturbation method in future work.

\paragraph*{Acknowledgment.}
The research was supported by  National Natural Science Foundation of
China (No. 61174043,61322311),  FANEDD, and Natural Science Foundation of
Jiangsu Province (No.BK2011582), NCET-11-0093.

\appendix
 \setcounter{equation}{0}
\renewcommand{\theequation}{A.\arabic{equation}}
\section{Proofs of the General Theorems on Discrete-time Stochastic Averaging}    \label{sec-proofs}
\subsection{Proof of Lemma \ref{lemma1}: approximation in finite-time interval with continuous average system }\label{sec-pro-lem1}


Fix $T>0$ and define
\begin{align}\label{proof-lem1-a}
M'=\sup_{0\leq t\leq T}|\bar{X}^{\rm c}(t)|.
\end{align}
Since $(\bar{X}^{\rm c}(t),t\geq 0)$ is continuous and $[0,T]$ is a compact set, we have that $M'<+\infty$. Denote $M=M'+1$. For any $\epsilon\in (0,\epsilon_0)$, define a stopping time $\tau_{\epsilon}$ by
\begin{align}\label{proof-lem1-b}
\tau_{\epsilon}=\inf\{t\geq 0:|X(t)|>M\}.
\end{align}
By the definition of $M$ (noting that $|x|=|X_0|=|\bar{X}^{\rm c}(0)|\leq M'$), we know that $0<\tau_{\epsilon}\leq +\infty$. If $\tau_{\epsilon}<+\infty$, then by the definition of $\tau_{\epsilon}$, we know that for any $s<\tau_{\epsilon}$, $|X(s)|\leq M$. By Assumption 1, we know that
there exists a positive constant $C_M$ such that for any $|x|\leq M$ and any $y$, we have $|f(x,y)|\leq C_M.$ And thus by (\ref{sys-1}), we know that
\begin{align}\label{proof-lem1-c}
M\leq |X({\tau_{\epsilon}})|\leq M+\epsilon C_M\leq M+\epsilon_0 C_M.
\end{align}
Denote $\bar{M}=M+\epsilon_0 C_M$. By Assumption 1 again, we know that
there exists a positive constant $C_{\bar{M}}$ such that for any $|x|\leq \bar{M}$ and any $y$, we have $|f(x,y)|\leq C_{\bar{M}}.$ It follows by (\ref{a-bar}) that for any $|x|\leq \bar{M}$, $|\bar{f}(x)|\leq C_{\bar{M}}$.

From (\ref{sys-2}) and (\ref{avesys-2'}), we have that, for any $t\geq 0$,
\begin{align}\label{proof-lem1-d}
X(t)-\bar{X}^{\rm c}(t)&=\int_0^t [f(X(s),Y(\epsilon+s))-\bar{f}(\bar{X}^{\rm c}(s))]ds\cr
&\quad-
\int^t_{t_{m(t)}}f(X(s),Y(\epsilon+s))ds\cr
&=\int_0^t [f(X(s),Y(\epsilon+s))-f(\bar{X}^{\rm c}(s),Y(\epsilon+s))]ds\cr
&\quad+
\int_0^t [f(\bar{X}^{\rm c}(s),Y(\epsilon+s))-\bar{f}(\bar{X}^{\rm c}(s))]ds
-\cr
&\quad\quad \int^t_{t_{m(t)}}f(X(s),Y(\epsilon+s))ds.
\end{align}
By Assumption 1 and the definition of $\bar{f}$, there exists a positive constant  $K_{\bar{M}}$ such that for any $x_1,x_2$ in the subset  $D_{\bar{M}}:=\{x\in \mathbb{R}^n:|x|\leq \bar{M}\}$ of $\mathbb{R}^n$, and any $y\in \mathbb{R}^m$, we have
\begin{align}
|f(x_1,y)-f(x_2,y)|&\leq K_{\bar{M}}|x_1-x_2|,\label{proof-lem1-e}\\
|\bar{f}(x_1)-\bar{f}(x_2)|&\leq K_{\bar{M}}|x_1-x_2|.\label{proof-lem1-e-1}
\end{align}
By (\ref{proof-lem1-d}), (\ref{proof-lem1-e}) and (\ref{proof-lem1-e-1}), we have that if $t\leq \tau_{\epsilon}\wedge T$, then
\begin{align}\label{proof-lem1-f}
|X(t)-\bar{X}^{\rm c}(t)|&\leq K_{\bar{M}}\int_0^t |X(s)-\bar{X}^{\rm c}(s)|ds\cr
&\quad+
\left|\int_0^t [f(\bar{X}^{\rm c}(s),Y(\epsilon+s))-\bar{f}(\bar{X}^{\rm c}(s))]ds\right|\cr
&\quad
+\left|\int_{t_{m(t)}}^t f(X(s),Y(\epsilon+s))ds\right|.
\end{align}
Define
\begin{align}
\Delta_t&=|X(t)-\bar{X}^{\rm c}(t)|,\label{proof-lem1-g}\\
\alpha(\epsilon)&=\sup_{0\leq t \leq T}\left|\int_0^t [f(\bar{X}^{\rm c}(s),Y(\epsilon+s))-\bar{f}(\bar{X}^{\rm c}(s))]ds\right|,\label{proof-lem1-h}\\
\beta(\epsilon)&=\sup_{0\leq t \leq\tau_{\epsilon}\wedge  T}\left|\int_{t_{m(t)}}^t f(X(s),Y(\epsilon+s))ds\right|\label{proof-lem1-i}.
\end{align}
Then by (\ref{proof-lem1-f}) and Gronwall's inequality, we have
\begin{align}\label{proof-lem1-j}
\sup_{0\leq t\leq \tau_{\epsilon}\wedge T}\Delta_t&\leq (\alpha(\epsilon)+\beta(\epsilon))e^{K_{\bar{M}}(\tau_{\epsilon}\wedge T)}\cr
&\leq (\alpha(\epsilon)+\beta(\epsilon))e^{K_{\bar{M}}T}.
\end{align}
Since for any $t\geq 0$, we have $t-t_{m(t)}\leq \epsilon$, and thus $\beta(\epsilon)\leq C_{\bar{M}}\epsilon$. Hence
\begin{align}\label{proof-lem1-k}
\lim_{\epsilon\to 0}\beta(\epsilon)=0.
\end{align}
In the following, we  prove
 that $\lim_{\epsilon\to 0}\alpha(\epsilon)=0\ a.s.$, i.e.
\begin{align}\label{proof-lem1-l}
\lim_{\epsilon\to 0}\sup_{0\leq t \leq T}\left|\int_0^t  [f(\bar{X}^{\rm c}(s),Y(\epsilon+s))-\bar{f}(\bar{X}^{\rm c}(s))]ds\right|=0\ a.s.
\end{align}
For any $n\in \mathbb{N}$, define a function $\bar{X}^n(s),s\geq 0$, by
\begin{align}\label{proof-lem1-m}
\bar{X}^n(s)=\sum_{k=0}^{\infty}\bar{X}^{\rm c}({k\over n})I_{\{{k\over n}\leq s<{k+1\over n}\}}.
\end{align}
Then for any $n\in \mathbb{N}$, we have
\begin{align}\label{proof-lem1-n}
\sup_{0\leq s\leq T}|\bar{X}^n(s)|\leq \sup_{0\leq s\leq T}|\bar{X}^{\rm c}(s)|= M'<\bar{M}.
\end{align}
By (\ref{proof-lem1-e}), (\ref{proof-lem1-e-1}), (\ref{proof-lem1-m}) and (\ref{proof-lem1-n}), we obtain that
\begin{align}\label{proof-lem1-o}
&\sup_{0\leq t \leq T}\left|\int_0^t  [f(\bar{X}^{\rm c}(s),Y(\epsilon+s))-\bar{f}(\bar{X}^{\rm c}(s))]ds\right|\cr
&=\sup_{0\leq t \leq T}\left|\int_0^t  \left\{[f(\bar{X}^{\rm c}(s),Y(\epsilon+s))-f(\bar{X}^n(s),Y(\epsilon+s))]\right.\right.\cr
&\quad\left.\left.+[f(\bar{X}^n(s),Y(\epsilon+s))-\bar{f}(\bar{X}^n(s))]\right.\right.\cr
&\quad\left.\left.+[\bar{f}(\bar{X}^n(s))-\bar{f}(\bar{X}^{\rm c}(s))]\right\}ds\right|\cr
&\leq \sup_{0\leq t \leq T}\int_0^t  \left|f(\bar{X}^{\rm c}(s),Y(\epsilon+s))-f(\bar{X}^n(s),Y(\epsilon+s))\right|ds\cr
&\quad+\sup_{0\leq t \leq T}\left|\int_0^t\left(f(\bar{X}^n(s),Y(\epsilon+s))-
\bar{f}(\bar{X}^n(s))\right)ds\right|\cr
&\quad+\sup_{0\leq t \leq T}\int_0^t\left|\bar{f}(\bar{X}^n(s))-\bar{f}(\bar{X}^{\rm c}(s))\right|ds\cr
&\leq 2K_{\bar{M}}T\sup_{0\leq t\leq T}|\bar{X}^{\rm c}(s)-\bar{X}^n(s)|\cr
&\quad+\sup_{0\leq t \leq T}\left|\int_0^t\left(f(\bar{X}^n(s),Y(\epsilon+s))-
\bar{f}(\bar{X}^n(s))\right)ds\right|.
\end{align}
Next, we focus on the second term on the right-hand side of (\ref{proof-lem1-o}). We have
\begin{align}\label{proof-lem1-p}
&\sup_{0\leq t \leq T}\left|\int_0^t\left(f(\bar{X}^n(s),Y(\epsilon+s))-
\bar{f}(\bar{X}^n(s))\right)ds\right|\cr
&=\sup_{0\leq t \leq T}\left|\int_0^t\left(f(\bar{X}^n(s),Y(\epsilon+s))-
\bar{f}(\bar{X}^n(s))\right)\sum_{k=0}^{\infty}I_{\{{k\over n}\leq s<{(k+1)\over n}\}}ds\right|\cr
&=\sup_{0\leq t \leq T}\left|\int_0^t\sum_{k=0}^{\infty}
\left(f(\bar{X}^{\rm c}({k\over n}),Y(\epsilon+s))-
\bar{f}(\bar{X}^{\rm c}({k\over n}))\right) I_{\{{k\over n}\leq s<{k+1\over n}\}}ds\right|\cr
&=\sup_{0\leq t \leq T}\left|\sum_{k=0}^{n([t]+1)}\int_{{k\over n}\wedge t}^{{(k+1)\over n}\wedge t}\left(f(\bar{X}^{\rm c}({k\over n}),Y(\epsilon+s))-\bar{f}(\bar{X}^{\rm c}({k\over n}))\right)ds\right|\cr
&\leq \sup_{0\leq t \leq T}\sum_{k=0}^{n([t]+1)}\left|\int_{{k\over n}\wedge t}^{{(k+1)\over n}\wedge t}\left(f(\bar{X}^{\rm c}({k\over n}),Y(\epsilon+s))-\bar{f}(\bar{X}^{\rm c}({k\over n}))\right)ds\right|,\nonumber\\
&
\end{align}
where $[t]$ is the largest integer not greater than $t$. For fixed $n$ and $k$ with $k\leq n([T]+1)$, we have
\begin{align}\label{proof-lem1-q}
&\sup_{0\leq t \leq T}\left|\int_{{k\over n}\wedge t}^{{k+1\over n}\wedge t}\left(f(\bar{X}^{\rm c}({k\over n}),Y(\epsilon+s))-
\bar{f}(\bar{X}^{\rm c}({k\over n}))\right)ds\right|\cr
&\leq \sup_{0\leq t \leq T}\left(\left|\int_0^{{k+1\over n}\wedge t}\left(f(\bar{X}^{\rm c}({k\over n}),Y(\epsilon+s))-
\bar{f}(\bar{X}^{\rm c}({k\over n}))\right)ds\right|\right.\cr
&\quad+\left.\left|\int_{0}^{{k\over n}\wedge t}\left(f(\bar{X}^{\rm c}({k\over n}),Y(\epsilon+s))-
\bar{f}(\bar{X}^{\rm c}({k\over n}))\right)ds\right|\right)\cr
&=2\sup_{0\leq t\leq {k+1\over n}}\left|\int_{0}^{t_{m(t)}}\left(f(\bar{X}^{\rm c}({k\over n}),Y(\epsilon+s))-
\bar{f}(\bar{X}^{\rm c}({k\over n}))\right)ds\right.\cr
&\quad \left.+
\int_{t_{m(t)}}^t\left(f(\bar{X}^{\rm c}({k\over n}),Y(\epsilon+s))-
\bar{f}(\bar{X}^{\rm c}({k\over n}))\right)ds\right|\cr
&\leq 2\sup_{0\leq t\leq {k+1\over n}}\left|\int_{0}^{t_{m(t)}}\left(f(\bar{X}^{\rm c}({k\over n}),Y(\epsilon+s))-
\bar{f}(\bar{X}^{\rm c}({k\over n}))\right)ds\right|\cr
&\quad\quad+4C_{\bar{M}}\epsilon.
\end{align}
For the second term on the right-hand side of (\ref{proof-lem1-q}), we have
\begin{align}\label{proof-lem1-r}
&\int_{0}^{t_{m(t)}}\left(f(\bar{X}^{\rm c}({k\over n}),Y(\epsilon+s))-
\bar{f}(\bar{X}^{\rm c}({k\over n}))\right)ds\cr
&=\epsilon\sum_{i=0}^{[t/\epsilon]-1}\left(f(\bar{X}^{\rm c}({k\over n}),Y(i+1))-
\bar{f}(\bar{X}^{\rm c}({k\over n}))\right)\cr
&=\epsilon[t/\epsilon]\frac{1}{[t/\epsilon]}
\sum_{i=0}^{[t/\epsilon]-1}\left(f(\bar{X}^{\rm c}({k\over n}),Y(i+1))-
\bar{f}(\bar{X}^{\rm c}({k\over n}))\right)\cr
&=\epsilon[t/\epsilon]\left(\frac{1}{[t/\epsilon]}
\sum_{i=0}^{[t/\epsilon]-1}f(\bar{X}^{\rm c}({k\over n}),Y(i+1))-
\bar{f}(\bar{X}^{\rm c}({k\over n}))\right)\cr
&
\end{align}
Then by  (\ref{proof-lem1-r}), the Birkhoff's ergodic theorem and \cite[Problem 5.3.2]{LipShi89}, we obtain that
\begin{align}\label{proof-lem1-s}
\lim_{\epsilon\to 0}\sup_{0\leq t\leq {k+1\over n}}&\left|\int_{0}^{t_{m(t)}}\left(f(\bar{X}^{\rm c}({k\over n}),Y(\epsilon+s))-
\bar{f}(\bar{X}^{\rm c}({k\over n}))\right)ds\right|\cr
&\quad\quad\quad=0\ \  a.s.,
\end{align}
which together with (\ref{proof-lem1-p}) and (\ref{proof-lem1-q}) implies that for any $n\in \mathbb{N}$,
\begin{align}\label{proof-lem1-t}
\lim_{\epsilon\to 0}\sup_{0\leq t \leq T}\left|\int_0^{t}\left(f(\bar{X}^n(s),Y(\epsilon+s))-
\bar{f}(\bar{X}^n(s))\right)ds\right|=0\ \ a.s.,
\end{align}
Thus by (\ref{proof-lem1-o}), (\ref{proof-lem1-t}) and
\begin{align}\label{proof-lem1-u}
\lim_{n\to\infty}\sup_{0\leq s\leq T}|\bar{X}^{\rm c}(s)-\bar{X}^n(s)|=0,
\end{align}
we obtain $\lim_{n\to\infty}\sup_{0\leq t \leq T}\left|\int_0^{t}  [f(\bar{X}^{\rm c}(s),Y(\epsilon+s))-\bar{f}(\bar{X}^{\rm c}(s))]\right.$ $ds\Big|=0\ \ a.s.$, i.e.
\begin{align}\label{proof-lem1-v}
\lim_{\epsilon\to 0}\alpha(\epsilon)=0\ \ a.s.
\end{align}

By (\ref{proof-lem1-g}), (\ref{proof-lem1-j}), (\ref{proof-lem1-k}) and
 (\ref{proof-lem1-v}),
we have
\begin{align}\label{proof-lem1-w}
\limsup_{\epsilon\to 0}\sup_{0\leq t\leq \tau_{\epsilon}\wedge
T}|X(t)-\bar{X}^{\rm c}(t)|=0\ \ a.s.
\end{align}
By (\ref{proof-lem1-a}) and (\ref{proof-lem1-w}), we have
\begin{align}\label{proof-lem1-x}
&\limsup_{\epsilon\to 0}\sup_{0\leq t\leq \tau_{\epsilon}\wedge
T}|X(t)|\cr
&\leq \limsup_{\epsilon\to 0}\left(\sup_{0\leq t\leq \tau_{\epsilon}\wedge
T}|X(t)-\bar{X}^{\rm c}(t)|+\sup_{0\leq t\leq \tau_{\epsilon}\wedge
T}|\bar{X}^{\rm c}(t)|\right)\cr
&\leq \limsup_{\epsilon\to 0}\sup_{0\leq t\leq \tau_{\epsilon}\wedge
T}|X(t)-\bar{X}^{\rm c}(t)|+M'\cr
&=M'<M\ a.s.
\end{align}
By (\ref{proof-lem1-c}) and (\ref{proof-lem1-x}), we obtain that, for almost every $\omega\in\Omega$, there exists an $\epsilon_0(\omega)$ such that for any $0<\epsilon<\epsilon_0(\omega)$,
\begin{equation}\label{proof-lem1-y}
\tau_{\epsilon}(\omega)>T.
\end{equation}
Thus by (\ref{proof-lem1-w}) and (\ref{proof-lem1-y}), we obtain that
\begin{equation}\label{proof-lem1-z}
\limsup_{\epsilon\to 0}\sup_{0\leq t\leq T}|X(t)-\bar{X}^{\rm c}(t)|=0\ a.s.
\end{equation}
Hence (\ref{lem-a}) holds. The proof is completed.

\subsection{Proof of Lemma \ref{lemma2}: approximation for finite-time interval with discrete average system}\label{sec-pro-lem2}
By Lemma 1, we need only to prove that
\begin{align}\label{proof-lem2-a}
\lim_{\epsilon\to 0}\sup_{0\leq t\leq
T}|\bar{X}^{\rm d}(t)-\bar{X}^{\rm c}(t)|=0.
\end{align}
Let $M',M,C_M,\bar{M},C_{\bar{M}},K_{\bar{M}}$ be defined in the above proof of Lemma 1.
For any $\epsilon\in (0,\epsilon_0)$, define a  time $\tau^{\rm d}_{\epsilon}$ by
\begin{align}\label{proof-lem2-b}
\tau^{\rm d}_{\epsilon}=\inf\{t\geq 0:|\bar{X}^{\rm d}(t)|>M\}.
\end{align}
By the definition of $M$ (noting that $|x|=|\bar{X}^{\rm d}_0|=|\bar{X}^{\rm c}(0)|\leq M'$), we know that $0<\tau^{\rm d}_{\epsilon}\leq +\infty$. If $\tau^{\rm d}_{\epsilon}<+\infty$, then by the definition of $\tau^{\rm d}_{\epsilon}$, we know that for any $s<\tau^{\rm d}_{\epsilon}$, $|\bar{X}^{\rm d}(s)|\leq M$.
By (\ref{avesys-1}), we know that
\begin{align}\label{proof-lem2-c}
M\leq |\bar{X}^{\rm d}({\tau^{\rm d}_{\epsilon}})|\leq M+\epsilon C_M\leq M+\epsilon_0 C_M.
\end{align}

Noting that if $t\leq \tau^{\rm d}_{\epsilon}\wedge T$, then
\begin{align}\label{proof-lem2-d}
|\bar{X}^{\rm c}(t)|\leq \bar{M},\ |\bar{X}^{\rm d}(t)|\leq \bar{M},
\end{align}
and for any $x_1,x_2$ in the subset  $D_{\bar{M}}:=\{x\in \mathbb{R}^n:|x|\leq \bar{M}\}$ of $\mathbb{R}^n$,  we have
\begin{align}\label{proof-lem2-e}
|\bar{f}(x_1)-\bar{f}(x_2)|&\leq K_{\bar{M}}|x_1-x_2|\ \mbox{and}\ |\bar{f}(x_1)|\leq C_{\bar{M}}.
\end{align}

By (\ref{avesys-2}) and (\ref{avesys-2'}), we have
\begin{align}\label{proof-lem2-f}
\bar{X}^{\rm d}(t)-\bar{X}^{\rm c}(t)=&\int_0^t (\bar{f}(\bar{X}^{\rm d}(s))-\bar{f}(\bar{X}^{\rm c}(s)))ds-\int^t_{t_{m(t)}}\bar{f}(\bar{X}^{\rm d}(s))ds.
\end{align}
Then by (\ref{proof-lem2-d})-(\ref{proof-lem2-f}) and the fact that $t-t_{m(t)}\leq \epsilon$, we obtain that for any $0\leq t\leq \tau^{\rm d}_{\epsilon}\wedge T$,
\begin{align}\label{proof-lem2-g}
|\bar{X}^{\rm d}(t)-\bar{X}^{\rm c}(t)|\leq K_{\bar{M}}\int_0^t |\bar{X}^{\rm d}(s))-\bar{X}^{\rm c}(s))|ds+C_{\bar{M}}\epsilon.
\end{align}
By (\ref{proof-lem2-g}) and the Gronwall's inequality, we get
\begin{align}\label{proof-lem2-h}
\sup_{0\leq t\leq \tau^{\rm d}_{\epsilon}\wedge T}|\bar{X}^{\rm d}(t)-\bar{X}^{\rm c}(t)|\leq C_{\bar{M}}\epsilon\exp(K_{\bar{M}}T),
\end{align}
which implies that
\begin{align}\label{proof-lem2-i}
\lim_{\epsilon\to 0}\sup_{0\leq t\leq \tau^{\rm d}_{\epsilon}\wedge T}|\bar{X}^{\rm d}(t)-\bar{X}^{\rm c}(t)|=0.
\end{align}
By (\ref{proof-lem1-a}) and (\ref{proof-lem2-i}), we have
\begin{align}\label{proof-lem2-j}
&\limsup_{\epsilon\to 0}\sup_{0\leq t\leq \tau^{\rm d}_{\epsilon}\wedge
T}|\bar{X}^{\rm d}(t)|\cr
&\leq \limsup_{\epsilon\to 0}\left(\sup_{0\leq t\leq \tau^{\rm d}_{\epsilon}\wedge
T}|\bar{X}^{\rm d}(t)-\bar{X}^{\rm c}(t)|+\sup_{0\leq t\leq \tau^{\rm d}_{\epsilon}\wedge
T}|\bar{X}^{\rm c}(t)|\right)\cr
&\leq \limsup_{\epsilon\to 0}\sup_{0\leq t\leq \tau^{\rm d}_{\epsilon}\wedge
T}|\bar{X}^{\rm d}(t)-\bar{X}^{\rm c}(t)|+M'\cr
&=M'<M.
\end{align}
By (\ref{proof-lem2-c}) and (\ref{proof-lem2-j}), we obtain that,  there exists an $\epsilon_0$ such that for any $0<\epsilon<\epsilon_0$,
\begin{equation}\label{proof-lem2-k}
\tau^{\rm d}_{\epsilon}>T.
\end{equation}
Thus by (\ref{proof-lem2-i}) and (\ref{proof-lem2-k}), we obtain that
\begin{equation}\label{proof-lem2-l}
\limsup_{\epsilon\to 0}\sup_{0\leq t\leq T}|\bar{X}^{\rm d}(t)-\bar{X}^{\rm c}(t)|=0.
\end{equation}
Hence (\ref{proof-lem2-a}) holds. The proof is completed.

\subsection{Proof of approximation results (\ref{thm3-a}) of Theorem  \ref{thm3}: approximation for any long time with continuous average system}\label{sec-pro-thm3-a}

Now we prove that for any $\delta>0$,
\begin{align}\label{dis-thm3-b}
\lim_{\epsilon\to0}\inf\{t\geq 0:|X(t)-\bar{X}^{\rm c}(t)|>\delta\}=+\infty \quad\mbox{a.s.}
\end{align}
Define
\begin{align}\label{dis-thm3-c}
\Omega'=\left\{\omega: \limsup_{\epsilon\to 0}\sup_{0\leq
t\leq T}|X(t,\omega)-\bar{X}^{\rm c}(t)|=0,\,\ \forall T\in
\mathbb{N}\right\},
\end{align}
where $X(t,\omega)$($=X(t)$) only makes the dependence on the sample clear.
Then by  Lemma \ref{lemma1}, we have
\begin{align}\label{dis-thm3-d}
P(\Omega')=1.
\end{align}
Let $\delta>0$. For   $\epsilon\in (0,\epsilon_0)$, define a
stopping time $\tau_{\epsilon}^{\delta}$ by
\begin{align}\label{dis-thm3-e}
\tau_{\epsilon}^{\delta}=\inf\{t\geq 0:
|X(t)-\bar{X}^{\rm c}(t)|>\delta\}.
\end{align}
By the fact that $X_0-\bar{X}_0=0$, and the right continuity of
the sample paths of $(X(t)-\bar{X}^{\rm c}(t),t\geq 0)$, we know that $0<\tau_{\epsilon}^{\delta}\leq +\infty$, and if
$\tau_{\epsilon}^{\delta}<+\infty$, then
\begin{align}\label{dis-thm3-f}
|X({\tau_{\epsilon}^{\delta}})-
\bar{X}^{\rm c}({\tau_{\epsilon}^{\delta}})|\geq \delta.
\end{align}
For any $\omega\in\Omega'$, by   (\ref{dis-thm3-c}) and (\ref{dis-thm3-f}),
we get that for any $T\in \mathbb{N}$, there exists
an $\epsilon_0(\omega,\delta,T)>0$ such that for any
$0<\epsilon<\epsilon_0(\omega,\delta,T)$,
\begin{align*}\label{dis-thm3-g}
\tau_{\epsilon}^{\delta}(\omega)>T,
\end{align*}
which implies that
\begin{equation}\label{dis-thm3-h}
\lim_{\epsilon\to 0}\tau_{\epsilon}^{\delta}(\omega)=+\infty.
\end{equation}
Thus it follows from (\ref{dis-thm3-d}) and (\ref{dis-thm3-h}) that
\begin{align*}\label{dis-lem2-i}
\lim_{\epsilon\to 0}\tau_{\epsilon}^{\delta}=+\infty\ a.s.
\end{align*}
The proof is completed.

\subsection{Proof of approximation results (\ref{thm3-b}) of Theorem  \ref{thm3}}\label{sec-pro-thm3-b}
The proof is similar to the proof (Appendixes C and D) of the continuous-time averaging results in \cite{LiuKrsTAC10} by replacing
$X^{\epsilon}_t$ and $\bar{X}_t$ with $X(t)$ and $\bar{X}^{\rm c}(t)$, respectively. The only difference lies in that $X(t)-\bar{X}^{\rm c}(t)$ is right continuous with respect to $t$, while both $X^{\epsilon}_t$ and $\bar{X}_t$ in \cite{LiuKrsTAC10} are continuous.

\subsection{Proof of  Theorem \ref{thm4}: the stability of the continuous-time version (\ref{sys-2}) of the original systems with the continuous average system}\label{sec-pro-thm4}
Since the equilibrium
$\bar{X}^{\rm c}(t)\equiv 0$ of continuous average system (\ref{avesys-20}) is exponentially stable,  there
exist constants $r>0,c>0$ and $\gamma>0$ such that for any $|x|<r$,
\begin{align*}
|\bar{X}^{\rm c}(t)|<c|x|e^{-\gamma t}, \ \ \forall t> 0.
\end{align*}
Thus for any $\delta>0$, we have
\begin{align*}
\left\{|X(t)|>c|x|e^{-\gamma t}+\delta
\right\}\subseteq\left\{|X(t)-\bar{X}^{\rm c}(t)|>\delta\right\},
\end{align*}
which together with  Theorem \ref{thm3} implies that
\begin{align*}
&\lim_{\epsilon\to 0}\ \inf\{t\geq 0:
|X(t)|>c|x|e^{-\gamma t}+\delta\}\cr
&\geq
\lim_{\epsilon\to 0}\ \inf\{t\geq 0:
|X(t)-\bar{X}^{\rm c}(t)|>\delta\}=+\infty\ a.s.
\end{align*}
Hence (\ref{thm4-5}) holds.\bigskip

Let $T(\epsilon)$ be defined in Theorem \ref{thm3}. Thus $\lim_{\epsilon\to 0}T_{\epsilon}=+\infty$. 
Since the equilibrium $\bar{X}^{\rm c}(t)=0$ of the average system is
exponentially stable,  there exist  constants $r>0,c>0$, and $\gamma>0$
such that for any $|x|<r$,
\begin{eqnarray}
|\bar{X}^{\rm c}(t)|<c|x|e^{-\gamma t}, \ \ \ \forall t>0.
\end{eqnarray}
Thus for any $\delta>0$, we have that for any $|x|<r$,
\begin{eqnarray}
&&\left\{\sup_{0\leq t\leq
T(\epsilon)}\left\{|X(t)|-c|x|e^{-\gamma t}\right\}>\delta
\right\}\cr
&&=\bigcup_{0\leq t\leq
T(\epsilon)}\left\{|X(t)|-c|x|e^{-\gamma
t}>\delta \right\}\nonumber\\
&& \subseteq\bigcup_{0\leq t\leq
T(\epsilon)}\left\{|X(t)-\bar{X}^{\rm c}(t)|>\delta\right\}\cr
&&=\left\{\sup_{0\leq t\leq
T(\epsilon)}|X(t)-\bar{X}^{\rm c}(t)|>\delta\right\},
\end{eqnarray}
which together with result (\ref{thm3-b}) of
Theorem \ref{thm3} gives that
\begin{eqnarray}
&&\limsup_{\epsilon\to 0}\ P\left\{\sup_{0\leq t\leq
T(\epsilon)}\left\{|X(t)|-c|x|e^{-\gamma t}\right\}>\delta
\right\}\cr &&\leq \lim_{\epsilon\to 0}P\left\{\sup_{0\leq t\leq
T(\epsilon)}|X(t)-\bar{X}^{\rm c}(t)|>\delta\right\}=0.
\end{eqnarray}
Hence (\ref{thm4-f}) holds. The  proof is
completed.

\subsection{Proof of Theorem \ref{thm5}}\label{sec-pro-thm5}

By using Lemma \ref{lemma2}, we can prove this theorem by following the proof of Theorem \ref{thm3}.
We omit the details.

\subsection{Proof of Theorem \ref{thm6}}\label{sec-pro-thm6}

By using Theorem \ref{thm5}, we can prove this theorem by following the proof of Theorem \ref{thm4}. We omit the details.

\subsection{Proof of Lemma \ref{lem7}}\label{sec-pro-lem7}
By Lemma \ref{lemma2} and the time scale transform, we get
\begin{eqnarray}\label{proof-lem7-a}
&&\limsup_{\epsilon\to 0}\sup_{0\leq k\leq
[N/\epsilon]}|X_k-\bar{X}^{\rm d}_k|\cr
&&=\limsup_{\epsilon\to 0}\sup_{0\leq k\leq
[N/\epsilon]}|X(\epsilon k)-\bar{X}^{\rm d}(\epsilon k)|\cr
&&=\limsup_{\epsilon\to 0}\sup_{0\leq t\leq
N}|X(t)-\bar{X}^{\rm d}(t)|=0\ a.s.
\end{eqnarray}
Hence (\ref{lem7-a}) holds. The proof is completed.

\subsection{Proof of Theorem \ref{thm8}}\label{sec-pro-thm8}

(i) Noticing that $[N/\epsilon]\geq N$ for $\epsilon\leq 1$. Then by Lemma \ref{lem7},  we know that for any natural number $N$,
\begin{eqnarray}\label{proof-thm8-a}
\lim_{\epsilon\to 0}\sup_{0\leq k\leq
N}|X_k-\bar{X}^{\rm d}_k|=0\ a.s.
\end{eqnarray}
By (\ref{proof-thm8-a}) and following the proof of Theorem \ref{thm3}(i), we can prove (\ref{thm8-a}).

(ii) By  Lemma \ref{lem7}, we know that for any natural number $N$, $\sup_{0\leq k\leq
N}|X_k-\bar{X}^{\rm d}_k|$ converges to 0 a.s., and thus it converges to 0 in probability, i.e. (\ref{thm8-b}) holds.

\subsection{Proof of Theorem \ref{thm9}}\label{sec-pro-thm9}

By using Theorem \ref{thm8}, we can prove this theorem by following the proof of Theorem \ref{thm4}. We omit the details.


\begin{thebibliography}{99}


\bibitem{AriKrs03}
K. B. Ariyur and M. Krsti\'{c}, {\it Real-Time Optimization by
Extremum Seeking Control}, Hoboken, NJ: Wiley-Interscience, 2003.

\bibitem{BaiFuSas88}
E. -W. Bai, L.-C. Fu, and S.S.Sastry, ``Averaging analysis for discrete time and
sampled data adaptive systems'', {\it IEEE Transactions
on Automatic Control}, vol. 35, no. 2, pp. 137-148, 1988.


\bibitem{BenMetPri90}
A. Benveniste, M. M\'{e}tivier, and P. Priouret, {\it Adaptive
Algorithms and Stochastic Approximations}, Springer-Verlag, 1990.




\bibitem{BogMit61}
N. N. Bogoliubov and Y. A. Mitropolsky, {\it Asymptotic Methods in
the Theory of Nonlinear Oscillation}, Gordon and Breach Science
Publishers INC, New York, 1961.

\bibitem{Che03}
H.-F. Chen, {\it Stochastic Approximation and Its Applications}, Kluwer Academic Publisher, 2003.


\bibitem{ChoKrsAriLee02}
J. -Y. Choi, M. Krstic, K. B. Ariyur, and J. S. Lee, ``Extremum seeking control
for discrete time systems'',  {\it IEEE Translation on Automatic and Control}, vol. 47, no. 2, pp. 318-323, 2002.


 \bibitem{FreWen84}
 M. I. Freidlin and A. D. Wentzell, {\it Random Perturbations of
Dynamical Systems}, Springer-Verlag, 1984.

\bibitem{GulVer93}
O. V. Gulinsky and A. Yu Veretennikov, {\it Large Deviations for Discrete-Time Processes with Averaging}, 1993.



\bibitem{Kha02}
H. K. Khalil, {\it Nonlinear Systems},  third edition, Prentice
Hall, 2002.


\bibitem{Kha80}
R. Z. Khas'minski\v{\i}, {\it Stochastic Stability of Differential
Equations}, Sijthoff \& Noordhoff, 1980.

\bibitem{KrsWan00}
M. Krstic and H. H. Wang, Stability of extremum seeking
feedback for general nonlinear dynamic systems, {\it Automatica},
vol. 36, pp. 595-601, 2000.

\bibitem{KusYin03}
H. J. Kushner and G. Yin, {\it Stochastic Approximation and Recursive Algorithms and Applications}, second edition, Springer-Verlag, 2003.

\bibitem{LipShi89}
R. S. Liptser and A. N. Shiryayev, {\it Theory of Martingales},
Kluwer Academic Publishers, 1989.


\bibitem{LiuKrsAuto10}
S. -J. Liu and M. Krstic, ``Stochastic source seeking for nonholonomic unicycle'', {\it Automatica}, vol.46, no.9, pp. 1443-1453, 2010.

\bibitem{LiuKrsTAC10}
S. -J. Liu and M. Krstic, ``Stochastic averaging in continuous time and its applications to extremum seeking'', {\it IEEE Transactions on Automatic Control}, vol.55, no.10, pp.2235-2250, 2010.


\bibitem{KhoTanManNes13}
S. Z. Khong, Y. Tan, C. Manzie, and D. Ne\v{s}i\'{c}, ``Unified frameworks for sampled-data extremum seeking control: Global optimization and multi-unit systems'', {\it Automatica}, vol. 49, pp. 2720-2733, 2013.
%

\bibitem{Lju77}
 L. Ljung, ``Analysis of recursive stochastic algorithms'', {\it IEEE
Transactions on Automatic Control}, vol. 22, pp. 551-575, 1977.





\bibitem{ManKrs09}
C. Manzie and M. Krsti\'{c}, ``Extremum seeking with stochastic
perturbations'', {\em IEEE Transactions on Automatic Control}, vol.
54, pp. 580--585, 2009.


\bibitem{MoaManBre10}
W. H. Moase, C. Manzie, and M. J. Brea,``Newton-like extremum-seeking for the control of thermoacoustic instability'', {\it IEEE Transactions
on Automatic Control}, vol.55, pp. 2094-2105, 2010.


\bibitem{OuXuSch07}
Y. Ou, C. Xu, E. Schuster, T. Luce, J.R.Ferron, and M. Walker, ``Exremum-seeking finite-time optimal control of plasma current profile at the DIII-D Tokamak'', {\it Proceedings of the 2007 American Control Conference}, July 11-13, pp. 4015-4020, 2007.


\bibitem{PopJanMagTee06}
D. Popovic, M. Jankovic, S. Magner, and A. Teel, ``Extremum seeking methods for optimization of variable cam timing ergine operation'', {\it IEEE Transactions on Control Systems  Technology}, vol. 14, no.3, pp.398-407, 2006.


\bibitem{RobSpa86}
J. B. Roberts and P. D. Spanos, ``Stochastic averaging: an
approximate method of solving random vibration problems'', {\it
International Journal Non-linear Mechanics}, vol, 21, no. 2,  pp.
111-134, 1986.

\bibitem{SanVerMur07}
J. A. Sanders, F. Verhulst, and J. Murdock, {\it Averaging Methods
in Nonlinear Dynamical Systems},  second edition, Springer, 2007.

\bibitem{SasBod89}
S. Sastry and M. Bodson, {\it Adaptive Control: Stability,
Convergence, and Robustness}, Prentice Hall, Englewood Cliffs, New
Jersey, 1989.

\bibitem{SchTorXu06}
E. Schuster, N. Torres and C. Xu, ``Extremum seeking adaptive
control of beam envelope in particle accelerators'', {\it
Proceedings of the 2006 IEEE Conference on Control Applications},
Munich, Germany, October 4-6, pp. 1837-1842, 2006.


\bibitem{Sko89}
  A. V. Skorokhod, {\it Asymptotic Methods in the Theory of
Stochastic Differential Equations}, Translations of Mathematical
Monographs, American Mathematical Society, 1989.


\bibitem{SkoHopSal02}  A. V. Skorokhod,  F. C.
Hoppensteadt and H. Salehi, {\it Random Perturbation Methods with
Applications in Science and Engineering}, Springer, 2002.

\bibitem{Soc98}
L. Socha,  ``Exponential stability of singularly perturbed stochastic systems'',
{\it IEEE Transactions on Automatic Control}, vol. 45, no.3, pp. 576-580, 2000.

\bibitem{SoloKong95}
V. Solo and X. Kong, {\it Adaptive Signal Processing Algorithms:
Stability and Performance}, Prentice Hall, 1995.

\bibitem{Spall03}
 J. C. Spall, {\it Introduction to Stochastic Search and
Optimization: Estimation, Simulation, and Control},
Wiley-Interscience, 2003.




\bibitem{StaSti09}
  M. S. Stankovic and D. M. Stipanovic, ``Discrete time extremum seeking by autonomous vehicles in stochastic environment'', {\it Proceedings of the 48th IEEE Conference on Decision and Control}, Shanghai, China, Dec.16-18, pp. 4541-4546, 2009.


 \bibitem{StaSti10}
 M. S. Stankovic and D. M. Stipanovic, ``Extremum seeking under stochastic noise and applications to mobile sensors'', {\it Automatica}, vol. 46, pp. 1243-1251, 2010.



\bibitem{TanNesMar06}
Y. Tan, D. Ne\v{s}i\'{c}  and I. Mareels, ``On non-local stability
properties of extremum seeking control'', {\it Automatica}, vol. 42,
no, 6, pp. 889-903, 2006.


 \bibitem{TeePop01}
A. R. Teel and D. Popovic, ``Solving smooth and nonsmooth multivariable extremum seeking problems by the methods of nonlinear programming'', {\it Proceedings of American Control Conference}, vol. 3, pp. 2394-2399, 2001.



\bibitem{ZhuYan97}
W. Q. Zhu and Y. Q. Yang, ``Stochastic averaging of quasi-non
integrable-Hamiltonian systems'', {\it Journal of Applied
Mechanics}, vol. 64,  pp. 157-164, 1997.


%






%


%
%
%
%
%
%
%
%


%







%





%




\end{thebibliography}
\end{document}